\documentclass{amsproc}
\usepackage{mathrsfs, euscript}
\usepackage{bbm}
\usepackage{amssymb}
\usepackage{txfonts}
\usepackage{amscd}
\usepackage{amsfonts,latexsym,amsmath,amsthm,amsxtra,}
\usepackage[bookmarks,colorlinks]{hyperref}
\usepackage[all,cmtip]{xy} 
\usepackage{color}
\usepackage{colordvi}
\usepackage{multicol}

\newcommand{\alg}{{\text{alg}}}

\def\an{\mathrm{an}}
\def\loc{\mathrm{loc}}

    \newcommand{\BA}{{\mathbb {A}}}

     \newcommand{\BN}{{\mathbb {N}}}
     \newcommand{\BP}{{\mathbb {P}}}
    \newcommand{\BQ}{{\mathbb {Q}}}

     \newcommand{\BZ}{{\mathbb {Z}}}

    \newcommand{\CI}{{\mathcal {I}}} 
    \newcommand{\CK}{{\EuScript {K}}} \newcommand{\CL}{{\EuScript {L}}}
     
    \newcommand{\CO}{{\mathcal {O}}} \newcommand{\CP}{{\EuScript {P}}}

    \newcommand{\CU}{{\EuScript {U}}}

    \newcommand{\RG}{{\mathrm {G}}} \newcommand{\RH}{{\mathrm {H}}}

    \newcommand{\RM}{{\mathrm {M}}} 
     \newcommand{\RP}{{\mathrm {P}}}

    \newcommand{\RU}{{\mathrm {U}}}

    \newcommand{\End}{{\mathrm{End}}} 
    \newcommand{\Gal}{{\mathrm{Gal}}} \newcommand{\GL}{{\mathrm{GL}}}
     
    \newcommand{\Ind}{{\mathrm{Ind}}}

    \newcommand{\fo}{{\mathfrak{o}}}

     \newcommand{\rank}{{\mathrm{rank}}}

    \newcommand{\Res}{{\mathrm{Res}}}

\newcommand{\delete}[1]{}

     \newcommand{\SL}{{\mathrm{SL}}}
     
     \newcommand{\Sym}{{\mathrm{Sym}}}

        \newcommand{\Sp}{{\mathrm{Sp}}}

    \newcommand{\ra}{\rightarrow}

    \theoremstyle{plain}

    \newtheorem{thm}{Theorem}[section] \newtheorem{cor}[thm]{Corollary}
    \newtheorem{lem}[thm]{Lemma}  \newtheorem{prop}[thm]{Proposition}
     \newtheorem{defn}[thm]{Definition}
    \newtheorem {rem}[thm]{Remark} 
    
    \numberwithin{equation}{section}

\begin{document}

\title{Morita's Theory for the Symplectic Groups}
\author{Zhi Qi and Chang Yang}
\begin{abstract}
We construct and study the holomorphic discrete series
representation and the principal series representation of the
symplectic group $\Sp(2n,F)$ over a $p$-adic field $F$ as well as a
duality between some sub-representations of these two
representations. The constructions of these two representations
generalize those defined in Morita {and Murase}'s works. Moreover, Morita built a
duality for $\SL(2, F)$ defined by residues. We view the duality  
constructed here as an algebraic interpretation of Morita's duality in some
extent and its generalization to the symplectic groups.
\end{abstract}

\keywords{symplectic groups, $p$-adic Siegel upper half-space,
principal series, holomorphic discrete series, Morita's duality,
Casselman's intertwining operator. }

\maketitle

{\footnotesize
\tableofcontents
}


\section*{Notations}

Let $p$ be a prime, $F$ be a finite extension of $\BQ_p$, $\fo$ be the
ring of integers of $F$, $\varpi$ be a uniformizer of $\fo$, $|\ |$
be the normalized absolute value, and $F^\alg$ be an algebraic closure of
$F$. Let $K$ be an extension of $F$ with an absolute value extending
$|\ |$ such that $K$ is complete under this absolute value. Because the
Hahn-Banach theorem is applied, we assume that $K$ is spherically
complete in \S \ref{sec:Duality} and \S \ref{sec:SL2Morita}.

\addtocontents{toc}{\protect\setcounter{tocdepth}{2}}

\addtocounter{section}{-1}
\section{Introduction}

\subsection*{Backgrounds}
In \cite{MoritaI}, Morita and Murase constructed and studied   
$p$-adic holomorphic discrete series representations of $ \SL(2,
F)$. In \cite{Schneider1992}, Schneider introduced the holomorphic
discrete series of $\SL(n+1, F)$ associated to a rational
representation of $\GL(n, F)$. He showed that, as   $\SL(n+1,
F)$-representation, the space of holomorphic exterior differential
$r$-forms on the Drinfe'ld's space is actually a  holomorphic
discrete series.

Morita started the systematic study of principal series
(parabolic induced representations) of $\SL(2, F)$ in
\cite{MoritaII} and \cite{MoritaIII}. In order to prove the
irreducibility conjectures on holomorphic discrete series,
Morita later constructed a duality pairing via residues between 
holomorphic discrete series and  principal series
(\cite{MoritaII}).



\subsection*{Outline of article}

In the first paragraph, we  generalize Morita's constructions to the
symplectic groups. In \S \ref{sec: the symplectic groups}, we     recollect  some notions on the symplectic
groups. In \S \ref{sec:principal},  following \cite{MoritaII}, we
give another interpretation of a parabolic induced
representation, which is conventionally called a  principal series.
General results of F\'eaux de Lacroix on   induced representations
of   $F$-Lie groups (\cite{Feaux}) are applied for our purpose. In
\S \ref{sec:Siegel},  a $p$-adic analogue of the Siegel
upper half-space, along with an $F$-rigid analytic
structure, is introduced. The method is similar to the one utilized in the study of
Drinfe'ld's space in \cite{Schneider-Stuhler}. In \S
\ref{sec:discrete}, we introduce the notion of the holomorphic
discrete series of $\Sp(2n, F)$ associated to a $K$-rational
representation of $\GL(n, F)$ and prove that the space of rigid
analytic exterior differential $r$-forms on the Siegel upper
half-space can be realized as a holomorphic discrete series.

In the second paragraph, in a purely algebraic way, we construct two
invariant closed subspaces of the principal series and the
holomorphic discrete series respectively and establish a duality
operator between them. We remark that, since the two spaces are of
compact type and nuclear $K$-Fr\'echet, respectively, the duality
fits into the framework of Schneider and Teitelbaum's theory (cf.
\cite{Schneider-Teitelbaum2002}).

In the last paragraph, in the case of $\SL(2, F)$, we analyze the
relations between the duality constructed in the second paragraph
and Morita's duality: composing with Casselman's intertwining
operator defined by differentiation, Morita's duality coincides with ours up to a constant.
\\
\\
\begin{small}
\emph{Acknowledgements.} We are especially grateful to Bingyong Xie.
He guided our researches and communicated many important ideas to
us. We also want to thank Professor P. Schneider for several
comments and advices.
\end{small}

\section{Symplectic groups and their representations}
\label{sec:SymGp}

\subsection{The symplectic group $\Sp(2n,F)$}\label{sec: the symplectic groups}

Let $n$ be a positive integer and $$J_n:=\begin{pmatrix}0 &I_n \\
-I_n & 0\end{pmatrix}. $$ The \emph{symplectic group} $\Sp(2n,F)$ is
the subgroup of $\GL(2n,F)$ that consists of matrices $g$ satisfying $$ {^tg}
J_n g =J_n.$$ If one writes $g=\begin{pmatrix}A &B \\ C &
D\end{pmatrix}$ ($A$, $B$, $C$, $D\in \RM(n, F)$), then $g
\in\Sp(2n,F)$ if and only if either one of the following two conditions
holds:
\begin{align}
& {^tA} D - {^tC}B=I_n,
\hskip 10pt {^tA} C = {^tC} A, \hskip 10pt {^tB}D ={^tD} B ; \label{eq:rel1}\\
& D\ {^tA} - C\ {^tB}=I_n, \hskip 4pt D\ {^tC}=C\ {^tD}, \hskip 2pt
B\ {^tA}=A\ {^tB}. \label{eq:rel2}
\end{align}

In the following, we introduce two homogeneous spaces $\CP(n)$ and $\CL(n)$ associated to $\Sp(2n, F)$.

Let $\CP(n)$ denote  the set of pairs $(X,Y)$, $X, Y \in \RM(n, F)$,
such that $$ X\ {^tY}=Y\ {^tX}, \hskip 10pt \rank{ (X \; Y)}=n.$$ 
A right action of $\Sp(2n,F)$ and a left action of $\GL(n,
F)$ on $\CP(n)$ are defined by
\begin{eqnarray*} & & (X,Y)g := (XA+YC, XB+YD), \hskip 10pt g=\begin{pmatrix}A &B
\\ C & D\end{pmatrix}\in \Sp(2n, F), \\
\nonumber & & h(X,Y) := (hX, hY), \hskip 10pt h\in\GL(n, F),
\end{eqnarray*} respectively. Let
$$\RU:=\left\{\begin{pmatrix} I_n & B \\0 & I_n
\end{pmatrix}\in \Sp(2n,F)\right\}.
$$ 
We identify the $\Sp(2n, F)$-homogeneous space $\RU\backslash \Sp(2n,F)$ with $\CP(n)$ via
$\begin{pmatrix}A &B \\ C & D\end{pmatrix} \mapsto (C, D)$ (the
inverse map comes from the symplectic Gram-Schmidt process).

Let $\CL(n)$ denote  the set of transposed Langrangian subspaces. Define
$$ \RP:=\left\{\begin{pmatrix}{^tD^{-1}} & B \\ 0 &
D\end{pmatrix}\in \Sp(2n,F)\right\}.$$ 
Then   $\CL(n)$ can be identified with   the
$\Sp(2n, F)$-homogeneous space $\RP\backslash \Sp(2n,F)$. Since
$\RP$ is a parabolic subgroup, $\RP\backslash \Sp(2n,F)$ is a smooth
projective variety over $F$.

In view of $\RP\cong \RU\rtimes \GL(n,F)$, we have a natural $\Sp(2n,
F)$-equivariant isomorphism $\GL(n,F)\backslash \CP(n)\cong \CL(n)$;
the projection from $\CP(n)$ onto $\CL(n)$ maps $(X,Y)$ to the
transposed Lagrangian subspace spanned by the row vectors of $(X \;
Y)$.

Finally, we define certain open subsets that define
the coordinates on $\Sp(2n, F)$, $\CP(n)$ and $\CL(n)$.

Let $$U_0 := \left\{ \begin{pmatrix}A &B \\
C & D\end{pmatrix} \in \Sp(2n,F)  :  \det(C) \neq 0 \right\}.$$
We have the following unique decomposition in $\Sp(2n, F)$ for
matrices in $U_0$:
\begin{equation}\label{eq:decompositionofU0}
\begin{pmatrix}A &B \\ C & D\end{pmatrix}
=\begin{pmatrix}I_n & AC^{-1} \\ 0 & I_n\end{pmatrix}\begin{pmatrix}
{^t}{C}{^{-1}} & 0 \\ 0 & C
\end{pmatrix} \begin{pmatrix}0 & -I_n \\
I_n & C^{-1}D
\end{pmatrix}.
\end{equation}
$AC^{-1}$ and $C^{-1}D$ are symmetric ((\ref{eq:rel1}) and
(\ref{eq:rel2})). Thus, one may identify $U_{0}$ with $\Sym(n,F)
\times \GL(n,F) \times \Sym(n,F)$.

Let
$\CU_0$ be the open subset of $\CP(n)$:
$$\big\{(h, hz) : h\in \GL(n,F), z\in \Sym(n,F)\big\}.$$
Under the identification $\CP(n) \cong \RU\backslash \Sp(2n,F)$,  we have $\CU_0 \cong
\RU\backslash U_{0}$.

Furthermore, we 
identify $\Sym(n,F)$ with the open subset $\RP\backslash
U_{0}$ of $\CL(n)$.

To lighten notations, hereafter we let $\RG= \Sp(2n,F)$, $\RG_{\fo} = \Sp(2n,
\fo)$, $\RH= \GL(n,F)$, $\RH_\fo= \GL(n,\fo)$ and abbreviate
$\CP(n)$ and $\CL(n)$ to $\CP$ and $\CL$, respectively.
Moreover, let $\mathrm{pr}^\RG_{\CP}$, $\mathrm{pr}^\RG_{\CL}$ and
$\mathrm{pr}^{\CP}_{\CL}$ denote the canonical projections.

\subsection{$\Ind_{\RP}^{\RG}\sigma$ and the principal series
$(C^{\an}_\sigma(\CP , V), T_{\sigma})$}\label{sec:principal}

Let $(V, \sigma)$ be a \emph{locally analytic representation} (cf.
\cite{Feaux} 3.1.5 and \cite {Schneider-Teitelbaum2002} \S 3) of
$\RH$ on a barreled locally convex Hausdorff $K$-vector space $V$,
which means that the orbit maps are $V$-valued locally analytic
functions; more precisely, for any $v \in V$ there exists a BH-space
$W$ of $V$ (that is, a Banach space $W$ together with a continuous
injection $W \hookrightarrow V$) such that $g \mapsto \sigma(g)v$
expands in a neighborhood of the unit element to a power series
with $W$-coefficients (cf. \cite{Feaux}).

Observe that $\sigma$ extends to a representation of $\RP$ via the projection
$$\RP \rightarrow \RH, \hskip 10pt
\begin{pmatrix}{^tD^{-1}} & B \\ 0 & D\end{pmatrix} \mapsto D.$$ We
consider the \emph{parabolic induced representation}
$\Ind_{\RP}^\RG\sigma$ whose underlying space is the space of
$V$-valued locally analytic functions $f$ on $\RG$ satisfying $$
f(pg)=\sigma(p)f(g), \hskip 10pt \text{for all } g \in \RG, p \in
\RP;
$$
$\RG$ acts by the right translation.

Because the homogeneous space $\CL $ is compact,
$\Ind_{\RP}^\RG\sigma$ is a locally analytic representation of $\RG$
(\cite{Feaux} 4.1.5).

\delete{
\begin{lem}[cf. \cite{Feaux} 4.3.1]
Let $G$ be a finite dimensional $F$-Lie group, $H $ a locally
analytic subgroup of $G$, $\sigma$ a locally analytic representation
of $H$ on a $K$-Banach space $V$, $s: H \backslash G \ra G $ a
locally analytic section. Then
\begin{align*} s^* : \Ind_H^G\sigma &\ra C^\an(H\backslash G, V)\\
f &\mapsto f\circ s,
\end{align*} is an isomorphism.
\end{lem}
}

Next, we give another description of $\Ind_{\RP}^\RG \sigma$. Let
$C^{\an}_\sigma(\CP , V)$ be the space of $V$-valued locally
analytic functions $\varphi$ on $\CP $ satisfying $$\varphi(hX,
hY)=\sigma(h) \varphi(X,Y), \hskip 10pt \text{ for all } (X,Y)\in
\CP \text{ and } h\in \RH.
$$

We define the \emph{principal series representation}
$(C^{\an}_\sigma(\CP , V), T_{\sigma})$ of $\RG$:
\begin{equation}\label{eq:defTsigma}(T_{\sigma}(g) \varphi) (X,
Y):= \varphi((X, Y)g).\end{equation}

\begin{lem} \label{lem:principalseries}~
	
(1) The representation $\Ind_{\RP}^\RG\sigma$
is (naturally) isomorphic to $(C^\an_\sigma(\CP , V), T_{\sigma})$.

(2) $\Ind_{\RP}^\RG\sigma$ is isomorphic to $C^{\an}(\CL , V)$.
\end{lem}
\begin{proof}
(1) From a locally analytic section $\bar{\iota}$ of
$\mathrm{pr}^\RG_{\CP}$, one obtains an isomorphism $\bar{\iota}^\circ:
\Ind_{\RU}^\RG\mathbf{1} \simeq C^{\an} (\CP , V), f\mapsto f\circ
\bar{\iota}$ (\cite{Feaux} 4.3.1). By restriction,
$\bar{\iota}^\circ$ induces an isomorphism between
$\Ind_{\RP}^\RG\sigma$ and $C^{\an}_\sigma(\CP , V)$, which is
independent on the choice of $\bar{\iota}$. $\RG$-equivariance is evident.

(2) A locally analytic section $\tilde{\iota}$ of
$\mathrm{pr}^\RG_{\CL}$ induces an isomorphism $\tilde{\iota}^\circ:
\Ind_{\RP}^\RG\sigma \simeq C^{\an}(\CL , V) $ (ibid.).
\end{proof}

Because $\CL $ is compact, $C^{\an}(\CL , V)$ is of compact type
(\cite{Schneider-Teitelbaum2002} Lemma 2.1). By
\cite{Schneider-Teitelbaum2002} Proposition 1.2, Theorem 1.3 and
\cite{Schneider} Proposition 16.10, we have the following corollary.

\begin{cor}\label{cor:compacttype} Suppose that $ B$ is a closed subspace of
$C^{\an}_\sigma(\CP , V)$. Then both $B $ and $C^{\an}_\sigma(\CP , V)/B$
are of compact type, in particular, they are reflexive,
bornological, and complete; $B^*_b$ and $(C^{\an}_\sigma(\CP ,
V)/B)^*_b$ are nuclear Fr\'echet spaces.
\end{cor}

For technical needs, we fix a finite disjoint open covering
$\{\overline{\CU}_{\kappa }\}_{\kappa} $ of $\CL $ satisfying: \\
1. $\Sym(n, \fo) \in\{\overline{\CU}_{\kappa }\}_{\kappa} $; \\
2. each $\overline{\CU}_{\kappa }$ is translated into $\Sym(n, \fo)$
by an element
$g_\kappa$ in $\RG$;\\
3. Let $\CU_{\kappa } :=
(\mathrm{pr}^\CP_\CL)^{-1}(\overline{\CU}_{\kappa })$. We define the
analytic local section $\iota_\kappa: \overline{\CU}_{\kappa }\ra
\CU_{\kappa }$ of $\mathrm{pr}^{\CP}_{\CL}$ to be the
$g^{-1}_{\kappa }$-translation of the section
$$\iota_0: \Sym (n, F) \ra \CU_0, \hskip 10pt  z \mapsto (1,-z). $$
All the $\iota_\kappa$ give rise to  a locally analytic section $\iota$ of
$\mathrm{pr}^{\CP}_{\CL}$.
Define $\CK  :=  \iota(\CL)$. 

If the locally analytic sections $\bar{\iota}$ and $\tilde{\iota}$
in the proof of Lemma \ref{lem:principalseries} are compatible with
$\iota$, in the sense that $\tilde{\iota} = \bar{\iota}\circ \iota$, then
Lemma \ref{lem:principalseries} implies that $\iota$ induces an
isomorphism
\begin{eqnarray}\label{eq:iotacirc} \iota^{\circ} : C^{\an}_\sigma(\CP , V) &\ra& C^{\an}(\CL
, V) \\
\nonumber \varphi &\mapsto & \varphi\circ \iota.
\end{eqnarray}

\delete{For $\eta \in C^{\an}(\CL, V)$, we define a $V$-valued
locally analytic function $\varphi_{\eta }$ on $\CL$:
$$\varphi_{\eta }(h\cdot \iota (z)) := \sigma(h)\eta (z),
 \hskip 10pt h\in \RH, z \in\CL. $$
The inverse of $\iota^*$ is
\begin{align*}
\Phi : C^{\an}(\CL , V) &\ra C^{\an}_\sigma(\CP , V)\\
 \eta &\mapsto \varphi_{\eta}.
\end{align*}
}

\subsection{The $p$-adic Siegel upper half-space}
\label{sec:Siegel}

In this section, we  first define a $p$-adic analogue of the Siegel upper
half-space, which also generalizes the $p$-adic upper half-plane (cf.
\cite{DasTei:Lecture}), and then discuss    some of their basic
properties.


Let $\mathbf{S}$ be the $F$-rigid analytic variety $\mathbf{Sym}(n)$
that is isomorphic to the affine space $ \BA^{{n(n+1)}/{2}}_{/F}$.
{The underlying space of $\mathbf{S}$ is $\Sym(n,
F^\alg)$ (strictly speaking, $\Sym(n,
F^\alg)/\Gal(F^\alg/F)$ (cf. \cite{Bosch1984}), but it is more convenient not to consider the Galois action in our situation).} 

\begin{defn} Let
$$\mathbf{\Sigma} :=\left\{ Z\in \mathbf{S}  : \det(XZ+Y)\neq 0 \text{ for any pair }
(X, Y)\in \CP \right\}.$$ $\mathbf{\Sigma} $ is called the {\bf
$p$-adic Siegel upper half-space}.
\end{defn}

Firstly, we show that $\mathbf{\Sigma}$ is nonempty.

\begin{lem}
If $Z$ is a diagonal matrix in $\mathbf{S}$ whose diagonal entry
$Z_{ii}$ is of absolute value {$|\varpi|^{ 1/ {(n+1)^{k_i}}}$, with distinct positive integers $k_i$}, then $Z\in
\mathbf{\Sigma}$.
\end{lem}
\begin{proof} One needs to show that $\det (XZ+Y) $ is non-vanishing for any pair $(X, Y)\in \CP$. Suitably multiplying a matrix
in $\RH$ on the left and a permutation matrix on the right of
$X$ and $Y$, and conjugating $Z$ by the same permutation matrix, we may assume that $X=\begin{pmatrix}I_r & \widetilde{X} \\
0 & 0\end{pmatrix}$, with $r$ being the
rank of $X$. Moreover, we write $Y$ and $Z$ in block matrices $\begin{pmatrix}Y_{1} & Y_{2} \\
Y_{3} & Y_{4}\end{pmatrix}$ and $\begin{pmatrix}Z_1 & 0 \\
0 & Z_2\end{pmatrix}$, respectively. It follows from $X\ {^tY}=Y\ {^tX}$   that $ Y_{3} +
Y_{4}\ {^t\widetilde{X}}=0 $ and $Y_{1}+Y_{2}\ {^t\widetilde{X}} $ is symmetric.
Then
\begin{align*}
XZ+Y &= \begin{pmatrix}Z_1 + Y_1 & \widetilde{X}Z_2+Y_2 \\
-Y_4\ {^t\widetilde{X}} & Y_4\end{pmatrix} \\
&= \begin{pmatrix}I_r & 0 \\
0 & Y_4\end{pmatrix}\begin{pmatrix}Z_1 + Y_1 + \widetilde{X}Z_2\ {^t\widetilde{X}} + Y_2\ {^t\widetilde{X}} & \widetilde{X}Z_2+Y_2 \\
0 & I_{n-r}\end{pmatrix}\begin{pmatrix}I_r & 0 \\
-{^t\widetilde{X}} & I_{n-r}\end{pmatrix}.
\end{align*}
Therefore $$\det (XZ+Y) = \det (Z_1 + \widetilde{X}Z_2\ {^t\widetilde{X}} + Y_1
+ Y_2\ {^t\widetilde{X}})\det Y_4.$$ In view of $\rank (X \; Y)=n$ and $Y_3 = - Y_4
{^t\widetilde{X}}$,   $Y_4$ is invertible, namely $\det Y_4 \neq 0$.
Clearly, the first determinant on the right is a nonzero polynomial
in $Z_{ii}$ with coefficients in $F$, and the degree of $Z_{ii}$ in
each term does not exceed $  n$. By the assumptions on $Z_{ii}$, the terms that
appear  in the polynomial are of distinct absolute values, so the determinant is nonzero. In conclusion, $\det (XZ+Y)
\neq 0$.
\end{proof}

In the following, we  endow $\mathbf{\Sigma} $ with a structure of
$F$-rigid analytic variety and show that $\mathbf{\Sigma} $ is an admissible open
subset of $\mathbf{S}$ and consequently an open rigid analytic
subspace of $\mathbf{S}$ (compare \cite{Schneider-Stuhler} \S 1
Proposition 1).

We define $ \CP_\fo=\mathrm{pr}^\RG_\CP (\RG_{\fo});$ $\CP_\fo$ is
compact. By Iwasawa's decomposition, $\RG=\RP\cdot \RG_{\fo}$, 
$\CP=\RH \cdot \CP_\fo$, and therefore
$$ \mathbf{\Sigma} = \left\{ Z\in \mathbf{S} : \det(XZ+Y)\neq 0 \text{ for any pair } (X, Y)\in
\CP_\fo \right\}.
$$

For $Z\in \mathbf{S}$, let
$$|Z|:=\max_{1\leqslant i\leqslant j\leqslant n}\left\{1, |Z_{ij}|\right\}.$$
For a nonnegative integer $m$ and a pair $(X, Y)\in\CP_\fo $, we
define $$ \mathbf{B}^-(m; X,Y) := \left\{ Z\in\mathbf{S} :
|\det(XZ+Y) | < |Z|^n\ |\varpi|^{nm} \right\}.$$

\begin{lem}\label{lem:CBmXY}
If $m$ is a nonnegative integer and $(X, Y), (X', Y')\in \CP_\fo$
such that $(X, Y)\equiv (h X',h Y')\mod \varpi^{nm+1}$ for some
$h\in \RH_ \fo$, then
$$\mathbf{B}^-(m;X,Y) =\mathbf{B}^-(m; X',Y').$$
\end{lem}
\begin{proof} Obviously $\mathbf{B}^-(m; X,Y) = \mathbf{B}^-(m; h X,h Y) $.
We may therefore assume that $(X, Y)\equiv (X', Y')\mod \varpi^{nm+1}$.

We choose $\lambda\in (F^\alg)^\times$ such that $|\lambda|=|Z|$. With the observations that $|\lambda^{-1} |\leqslant 1$ and $|\lambda^{-1} Z_{ij}|\leqslant 1$,
one has
\begin{align*}&  X\cdot\lambda^{-1}Z+ Y\cdot\lambda^{-1} \equiv X'\cdot\lambda^{-1}Z+Y'\cdot\lambda^{-1} \mod \varpi^{nm+1},\\
& \  \det(X Z+Y ) \cdot \lambda^{-n} \equiv \det(X'Z+Y'
)\cdot\lambda^{-n} \mod \varpi^{nm+1},
\end{align*}
whence $$ |\det(XZ+Y)|\ |Z|^{-n} <|\varpi|^{nm} \Leftrightarrow
|\det(X'Z+Y')|\ |Z|^{-n} <|\varpi|^{nm}.$$ Therefore
$\mathbf{B}^-(m;X,Y) =\mathbf{B}^-(m; X',Y') $.
\end{proof}

Define
$$\mathbf{\Sigma}(m; X,Y) :=\mathbf{S} -  \mathbf{B}^-(m; X,Y) = \left\{ Z \in \mathbf{S} :
|\det(XZ+Y)| \geqslant  |Z|^n\ |\varpi|^{nm} \right\}.$$

Let
\begin{align*} \mathbf{\Sigma} (m) &:= \bigcap_{(X, Y)\in\CP_\fo} \mathbf{\Sigma} (m; X, Y) \\
&=  \left\{Z\in \mathbf{S} :
\left|\frac{\varpi^{nm}}{\det(XZ+Y)}\right|\leqslant 1,
\left|\frac{\varpi^{nm}Z_{i j}^n}{\det(XZ+Y)}\right|\leqslant 1 \text{
for any } (X, Y)\in\CP_\fo \right\}.
\end{align*}

{Let $\CP^{(m)}$ be any finite subset of $\CP_\fo$ containing $(0, I_n)$ as well as a set of representatives in
$\CP_\fo$ for $\RH_ \fo \backslash \CP_\fo\ (\mathrm{mod}\ 
\varpi^{nm+1})$. Then Lemma \ref{lem:CBmXY} implies that
$$\mathbf{\Sigma}(m) = \bigcap_{(X, Y)\in\CP^{(m)}} \mathbf{\Sigma} (m; X, Y).$$
Denote $\CP^{(m)}_0 = \CP^{(m)} - \{(0, I_n)\}$.}

Observe that 
\begin{align*}\mathbf{\Sigma}(m; 0,I_n)&=\left\{Z\in \mathbf{S}  
:  |Z_{ij}|\leqslant |\varpi|^{-m}, 1\leqslant i \leqslant j \leqslant n\right\}\\
&=\mathrm{Sp}\left(F\left\langle \varpi^{m}Z_{i j} :  1\leqslant i
\leqslant j \leqslant n \right\rangle\right),
\end{align*} is an admissible open affinoid subset of
$\mathbf{S}$. Thus, $\mathbf{\Sigma}(m)$ is the intersection of a finite
number of rational sub-domains of $\mathbf{\Sigma}(m; 0,I_n) $:
$$\left\{Z\in \mathbf{\Sigma}(m; 0,I_n) :
\left|\frac{\varpi^{nm}}{\det(XZ+Y)}\right|\leqslant 1,
\left|\frac{\varpi^{nm}Z_{i j}^n}{\det(XZ+Y)}\right|\leqslant 1 \right\},
$$ for all $(X, Y) \in \CP^{(m)}_0$. Therefore $\mathbf{\Sigma}(m) $ is the affinoid variety:
\begin{equation}\label{eq:Omegam}  \mathrm{Sp}\left(F
\left\langle \varpi^{m} Z_{ij} \right\rangle \left\langle
\frac{\varpi^{nm}}{\det(XZ+Y)}, \frac{\varpi^{nm}
Z_{ij}^n}{\det(XZ+Y)}  : (X, Y)\in \CP^{(m)}_0
\right\rangle\right).\end{equation}

Now, $\{\mathbf{\Sigma}(m)\}_{m=0}^{\infty}$ forms an admissible affinoid
covering of $\mathbf{\Sigma}$. This gives rise to a   rigid
analytic variety structure on $\mathbf{\Sigma}$ (see \cite{Bosch1984} 9.3). According to
\cite{Bosch1984} 9.1.2 Lemma 3 (compare \cite{Bosch1984} 9.1.4
Proposition 2), the following Proposition implies that
$\mathbf{\Sigma}$ is an admissible open subset of $\mathbf{S}$.

\begin{prop}\label{prop:Omegaadmissible}
Any morphism from an affinoid variety to $\mathbf{S}$ with image in
$\mathbf{\Sigma}$ factors through some $\mathbf{\Sigma}(m)$.
\end{prop}
\begin{proof}
The argument is similar to the third proof of
\cite{Schneider-Stuhler} \S 1 Proposition 1.

Let $\mathbf{X}$ be an affinoid variety and $\phi:
\mathbf{X}\rightarrow \mathbf{S}$ be a morphism from $\mathbf{X}$ to
$\mathbf{S}$ with image in $\mathbf{\Sigma}$. For any $(X,
Y)\in\CP_\fo$,  
$$x \mapsto \frac{1}{\det(X\phi(x)+Y)}, \hskip 10pt
x \mapsto \frac{(\phi(x))_{i j}^n}{\det(X\phi(x)+Y)}$$ are $F$-rigid
analytic functions on $\mathbf{X}$. By the maximum modulus principle
(\cite{Bosch1984} \S 6.2 Proposition 4 (i)), there exists a positive
integer $m_{ (X, Y)}$ such that
$$\max_{1\leqslant i\leqslant j\leqslant n}  \max_{x\in \mathbf{X}}\left\{
 \left|\frac{1}{\det(X\phi(x)+Y)}\right|, \left|\frac{(\phi(x))_{i j}^n}
 {\det(X\phi(x)+Y)}\right|\right\} \leqslant |\varpi|^{-n m_{(X, Y)}}.$$
{In other words}, $\phi(\mathbf{X})\subset \mathbf{\Sigma}(m_{ (X, Y)}; X,
Y)$. In view of Lemma \ref{lem:CBmXY}, one can choose $m_{(X,Y)}$ to be  locally constant. Therefore, there exists a
positive integer $m$ such that $\phi(\mathbf{X}) \subset
\mathbf{\Sigma}(m)$ due to  the compactness of    $\CP_\fo$.
\end{proof}

Let $\mathscr{O}(\mathbf{\Sigma}(m))$ denote the space of $F$-rigid
analytic functions on $\mathbf{\Sigma}(m)$; it is an $F$-affinoid
algebra with the supremum norm. From (\ref{eq:Omegam}) one sees  that
$\psi \in \mathscr{O}(\mathbf{\Sigma}(m))$ admits an expansion in the
form:
\begin{equation}\label{eq:psiZexpansion} \psi (Z) = \sum _{(k_{(X, Y)})\in
\BN^{\CP^{(m)}}} P_{(k_{(X, Y)})}(Z)\prod_{(X, Y)\in \CP^{(m)}}
\det(XZ+Y)^{-k_{(X, Y)}},\end{equation} where $\BN$ denotes the
set of nonnegative integers, $P_{(k_{(X, Y)})}(Z)$ are polynomials
in $Z_{ij}$ with $F$-coefficients, and the expansion converges
with respect to the supremum norm $\|\
\|_{\mathscr{O}(\mathbf{\Sigma}(m))}$. In particular,
$\det(XZ+Y)^{-1} \in \mathscr{O}(\mathbf{\Sigma}(m))$ for any $(X,
Y) \in \CP$. Let $\mathscr{O}(\mathbf{\Sigma})$ be the $F$-algebra
of $F$-rigid analytic functions on $\mathbf{\Sigma}$, which is the
projective limit of $\mathscr{O}(\mathbf{\Sigma}(m))$,
$$\mathscr{O}(\mathbf{\Sigma}) :=\underset{m}{\varprojlim}\mathscr{O}(\mathbf{\Sigma}(m)).
$$ We endow $\mathscr{O}(\mathbf{\Sigma})$ with the projective limit topology.

Let $\mathscr{O}_K(\mathbf{\Sigma}(m))$ and
$\mathscr{O}_K(\mathbf{\Sigma})$ denote
$\mathscr{O}(\mathbf{\Sigma}(m))\underset{F}{\otimes}  K $ and
$\mathscr{O}(\mathbf{\Sigma})\underset{F}{\otimes} K$, respectively.
If one let $\mathbf{\Sigma}_K(m)$ and $\mathbf{\Sigma}_K $ denote the
extension of the ground field $K/F$ of $\mathbf{\Sigma} (m)$ and
$\mathbf{\Sigma} $, respectively (\cite{Bosch1984} \S 9.3.6), then
$\mathscr{O}_K(\mathbf{\Sigma}(m))$ and
$\mathscr{O}_K(\mathbf{\Sigma})$ are the $K$-rigid analytic
functions on $\mathbf{\Sigma}_K(m)$ and $\mathbf{\Sigma}_K$,
respectively.

\begin{prop} \label{prop:OOmega}
$\mathscr{O}_K(\mathbf{\Sigma})$ is a nuclear $K$-Fr\'echet space.
\end{prop}
\begin{proof}
According to \cite {Schneider} Corollary 16.6 and Proposition 19.9, it
suffices to prove that $\mathscr{O}_K(\mathbf{\Sigma}(m))$ form a
compact projective system.

Observe that $\mathscr{O}_K(\mathbf{\Sigma}(m))$ is generated by
\begin{equation}\label{eq:generators}\varpi^{m} Z_{ij}, \frac{\varpi^{nm}}{\det(XZ+Y)},
\frac{\varpi^{nm} Z_{ij}^n}{\det(XZ+Y)}, \hskip 10pt 1\leqslant i\leqslant 
j\leqslant n, (X, Y)\in \CP_0^{(m)}.\end{equation} Since
$$\sup_{Z\in \mathbf{\Sigma}(m-1)} \sup_{ \scriptstyle (X, Y)\in
\CP_0^{(m)}\atop \scriptstyle 1\leqslant i \leqslant j \leqslant n } \left\{
|\varpi^{m} Z_{ij}|, \left|\frac{\varpi^{nm}}{\det(XZ+Y)} \right|,
\left| \frac{\varpi^{nm} Z_{ij}^n}{\det(XZ+Y)}\right| \right\}\leqslant 
|\varpi|.$$ By \cite {Schneider-Teitelbaum2002BV} Lemma 1.5, the
transition homomorphism from $\mathscr{O}_K (\mathbf{\Sigma}(m))$ to
$ \mathscr{O}_K (\mathbf{\Sigma}(m-1))$ is compact.
\end{proof}

Clearly a compact projective system passes to closed subspaces.
Moreover, a $K$-Fr\'echet space is the strong dual of a space of
compact type if and only if it is nuclear
(\cite{Schneider-Teitelbaum2002} Theorem 1.3).

\begin{cor}
Let $\mathscr{N}$ be a closed subspace of
$\mathscr{O}_K(\mathbf{\Sigma})$, then $\mathscr{N}$ is a nuclear
Fr\'echet space; $\mathscr{N}^*_b$ is of compact type.
\end{cor}

\begin{rem}
If $K$ is spherically complete, then Theorem 1.3 and Proposition 1.2 in
\cite{Schneider-Teitelbaum2002} imply that
$\mathscr{O}_K(\mathbf{\Sigma})/\mathscr{N}$ is also a nuclear
Fr\'echet space.
\end{rem}

All the generators (\ref{eq:generators}) of
$\mathscr{O}(\mathbf{\Sigma}(m))$ are $F$-rigid analytic functions
on $\mathbf{\Sigma}(m')$, for any $m' \geqslant m$, and therefore on $\mathbf{\Sigma}$ as well. Then we obtain the following proposition.

\delete{Moreover, observing that $\mathbf{\Sigma}(m-1)$ is a
Weierstra\ss \ domain of $\mathbf{\Sigma}(m)$ (if we take
$\CP^{(m-1)} = \CP^{(m)}$), the image of
$\mathscr{O}(\mathbf{\Sigma}(m))$ under the transition homomorphism
in $\mathscr{O}(\mathbf{\Sigma}(m-1))$ is dense, and therefore}

\begin{prop} \label{prop:Stein} \ 
	
(1) $\mathbf{\Sigma}$ is a Stein space (cf. \cite{Kiehl}), that is, the
image of $\mathscr{O}(\mathbf{\Sigma}(m + 1))$ under the transition
homomorphism in $\mathscr{O}(\mathbf{\Sigma}(m))$ is dense for any
nonnegative integer $m$.

(2) The image of $\mathscr{O}(\mathbf{\Sigma})$ under the transition
homomorphism in $\mathscr{O}(\mathbf{\Sigma}(m))$ is dense.
\end{prop}

Finally, we define a rigid analytic $\RG$-action  on $\mathbf{\Sigma}$:
$$gZ:=(AZ+B)(CZ+D)^{-1},\; g=\begin{pmatrix}A & B \\
C & D\end{pmatrix}\in \RG, Z \in \mathbf{\Sigma}. $$ 
(\ref{eq:rel1}) is required to show that this is indeed a $\RG$-action. Moreover, we define the
automorphy factor $$j(g, Z):= (CZ+D).$$ From a straightforward
computation, one verifies the automorphy (cocycle) relation
\begin{equation}\label{eq:automorphyfactor} j(g_1g_2, Z) = j(g_1,
g_2Z)j(g_2, Z).\end{equation}

\begin{lem} Let $m $ be a nonnegative integer. Then for any $g\in \RG_{\fo}$,{
$$g \mathbf{\Sigma}(m) \subset \mathbf{\Sigma}(nm).$$}\label{lem:gOmegam}
\end{lem}
\begin{proof}
Let $g=\begin{pmatrix}A & B \\ C & D\end{pmatrix}\in \RG_{\fo}$,
$Z\in \mathbf{\Sigma}(m)$ and $(X, Y) \in \CP_\fo$. We have $(XA+YC,
XB+YD)\in \CP_\fo$, whence
\begin{align*}& \frac {|Z|^{n}} {\left|\det (CZ+D) \right|}\leqslant |\varpi|^{-n
m} , \\
& \frac {|Z|^{n}} {\left|\det\big((XA+YC)Z+(XB+YD)\big)\right|}
\leqslant |\varpi|^{-n m}.\end{align*} 
By Cramer's rule,
\begin{align*}
gZ&= (AZ+B)(CZ+D)^{-1}
\\&= (AZ+B)\cdot \mathrm{adj}(CZ+D)\frac 1
{\det(CZ+D)}, \end{align*} where $\mathrm{adj}(CZ+D)$ denotes the
adjugate matrix of $CZ+D$. It is clear that $\det(CZ+D)$ and all the
entries of $(AZ+B)\cdot \mathrm{adj}(CZ+D)$ are polynomials in
$Z_{ij}$ with coefficients in $\fo$ and degree $\leqslant n$, so
\begin{align*} |\det(CZ+D)|& \leqslant |Z|^n, \\
|\left((AZ+B)\cdot \mathrm{adj}(CZ+D)\right)_{ij}| &\leqslant |Z|^n.\end{align*}
Combining these, 
\begin{align*}
& \frac {|g Z|^n} {\left|\det(X(g Z)+Y)\right|} \\
& =  \frac {\left|\det(CZ+D)\right|}
{\left|\det\big((XA+YC)Z+(XB+YD)\big)\right|} \max \left\{ 1 ,
\frac {|\left((AZ+B)\cdot \mathrm{adj}(CZ+D)\right)_{ij}|^n} {\left|\det(CZ+D)\right|^n}\right\}\\
&\leqslant \frac {|Z|^{n}} {\left|\det\big((XA+YC)Z+(XB+YD)\big)\right|}
\max \left\{ 1 , \frac {|Z|^{n^2-n}}
{\left|\det(CZ+D)\right|^{n-1}}  \right\}\\
&\leqslant |\varpi|^{-n^2m}.\end{align*} Therefore $g
\mathbf{\Sigma}(m)\subset \mathbf{\Sigma}(n m).$
\end{proof}

\subsection{Holomorphic discrete series $(\mathscr{O}_\sigma(\Sigma), \pi_{\sigma})$}
\label{sec:discrete} We abbreviate $\mathbf{\Sigma}_K(m)(K)$ and
$\mathbf{\Sigma}_K\\(K)$ to $\Sigma(m)$ and $\Sigma$, respectively.
Conventionally, $\mathscr{O}_K(\mathbf{\Sigma}(m))$ is described as
the space of $K$-valued functions on $\Sigma(m)$ with expansions of
the form (\ref{eq:psiZexpansion}) that converge in the supremum
norm of the $K$-valued function space on $\mathbf{\Sigma}_K(m) $,
and $\mathscr{O}_K(\mathbf{\Sigma})$ as the space of $K$-valued
functions on $\Sigma$ whose restrictions on $\Sigma(m)$ are
functions in $\mathscr{O}_K(\mathbf{\Sigma}(m))$. Furthermore, we abbreviate
$\mathscr{O}_K(\mathbf{\Sigma}(m))$ and
$\mathscr{O}_K(\mathbf{\Sigma})$ to $\mathscr{O}(\Sigma(m))$ and
$\mathscr{O}(\Sigma)$, respectively.

Let $(V,\sigma)$ be a $d$-dimensional $K$-rational representation of
$\RH$. Let $$\sigma(h) = \det(h)^{-s} \\ P(h), \hskip 10pt s \in
\BN, P \in \RM(d, K[h_{ij}]).$$

Let $\mathscr{O}_\sigma(\Sigma(m)) :=\mathscr{O}(\Sigma(m))
\underset{ K}{\otimes} V$ and $\mathscr{O}_\sigma(\Sigma)
:=\mathscr{O}(\Sigma)\underset{ K}{\otimes} V$. We define the
\emph{holomorphic (rigid analytic) discrete series representation}
$(\mathscr{O}_\sigma(\Sigma), \pi_{\sigma})$ of $\RG$:
\begin{equation}\label{eq:discrete}(\pi_{\sigma}(g) \psi) (Z) :=\sigma(j(g^{-1}, Z))^{-1}\psi(g^{-1}Z),
\hskip 10pt \psi\in \mathscr{O}_\sigma(\Sigma), g \in \RG.
\end{equation} 
Since Proposition \ref{prop:Omegaadmissible}
implies that $g^{-1}$ translates $\Sigma(m)$ into some $\Sigma(m')$,
it is not difficult to see that $\pi_{\sigma}(g) \psi\in
\mathscr{O}_\sigma(\Sigma)$. This follows from showing that its coordinates have expansions of the form
(\ref{eq:psiZexpansion}) and   are bounded under the supremum norms $\|\ \|_{\mathscr{O}(\Sigma(m))}$. By the automorphy relation (\ref{eq:automorphyfactor}), one verifies that $\pi_\sigma$ is a
$\RG$-representation.

\begin{prop}\label{prop:picontinuous}
$(\mathscr{O}_\sigma(\Sigma), \pi_{\sigma})$ is continuous.
\end{prop}

\begin{proof}
Since $\mathscr{O}_\sigma(\Sigma)$ is the projective limit of
$\mathscr{O}_\sigma(\Sigma(m))$, it suffices to prove, for each $m$,
the continuity of
\begin{eqnarray*}  \RG_{\fo} \times \mathscr{O}_\sigma(\Sigma)  & \rightarrow &
\mathscr{O}_\sigma(\Sigma(m))  \\    (g,\psi) & \mapsto &
(\pi_{\sigma}(g) \psi)|_{\Sigma(m)}.
\end{eqnarray*}
Moreover, according to Lemma \ref{lem:gOmegam}, $\RG_{\fo}
\mathbf{\Sigma}(m)\subset \mathbf{\Sigma}(nm)$, whence the above map
factors through $\RG_{\fo} \times \mathscr{O}_\sigma(\Sigma(nm))$.
Thus we only need to consider the continuity of the map:
\begin{eqnarray} \label{eq:GOtoOcontinuous} \RG_{\fo} \times \mathscr{O}_\sigma(\Sigma(nm))  & \rightarrow &
\mathscr{O}_\sigma(\Sigma(m))  \\
\nonumber (g,\psi) & \mapsto & (\pi_{\sigma}(g) \psi)|_{\Sigma(m)}.
\end{eqnarray}
For $g \in \RG_{\fo}$,   entries of $\sigma(j(g^{-1}, Z))^{-1}$
are $\fo$-coefficient polynomials, with  $Z_{ij}$,
$\det(j(g^{-1}, Z))^{-1}$ and the coefficients of $P$ viewed as variables. For $Z\in \mathbf{\Sigma}(m)$, we have $|Z_{ij}|\leqslant |\varpi|^{-m}$ and
$|\det(j(g^{-1}, Z))^{-1}| \leqslant |\varpi|^{-nm}$, so there is a
constant $c>0$ such that
$$\max_{g\in \RG_{\fo}} \max_{ Z\in
\mathbf{\Sigma}(m)}\|\sigma(j(g^{-1}, Z))^{-1}\|_{\End(V)}\leqslant c.
$$
Therefore

\begin{equation}\label{eq:continuous estimate}
\begin{split}
& \quad \max_{g\in \RG_{\fo}} \|\pi_{\sigma}(g) \psi\|_{\mathscr{O}_\sigma(\Sigma(m))}\\
&= \max_{g\in \RG_{\fo}} \max_{Z\in \mathbf{\Sigma}(m) } \|(\pi_{\sigma}(g) \psi)(Z)\|_V \\
&\leqslant  \max_{g\in \RG_{\fo}}\max_{Z\in\mathbf{\Sigma}(m) }
\|\sigma(j(g^{-1}, Z))^{-1}\|_{\End(V)} \cdot
\max_{g\in \RG_{\fo}} \max_{Z\in \mathbf{\Sigma}(m) } \|\psi(g^{-1} Z)\|_V \\
&\leqslant  c \max_{Z \in \mathbf{\Sigma}(nm) } \|\psi(Z)\|_V \\
&=  c \|\psi\|_{\mathscr{O}_\sigma(\Sigma(nm))}.
 \end{split}
\end{equation}
So the map (\ref{eq:GOtoOcontinuous}) is continuous.
\end{proof}

Now let $U_0(\fo)$ denote the parameterized open
neighborhood of the unit element, $ \Sym(n,\fo) \times \RH_\fo
\times \Sym(n,\fo)\subset U_0 \cap \RG_\fo$.

\begin{prop}\label{prop:pi_sigma analytic} For any $\psi\in \mathscr{O}_\sigma(\Sigma(nm))$, the orbit
map
\begin{eqnarray*}  U_0(\fo)  & \rightarrow & \mathscr{O}_\sigma(\Sigma(m))  \\
g & \mapsto & (\pi_{\sigma}(g) \psi)|_{\Sigma(m)}
\end{eqnarray*} is an $\mathscr{O}_\sigma(\Sigma(m))$-valued analytic function (that is, can be expanded as a convergent power series with variables the
coordinate parameters of $U_0(\fo)$ and coefficients in
the Banach space $\mathscr{O}_\sigma(\Sigma(m))$).
\end{prop}

\begin{proof} We first prove the following lemma.
\begin{lem}\label{lem:expansion} Let $\psi\in \mathscr{O}_\sigma(\Sigma(nm))$, $z\in \Sym(n, \fo)$ and $h\in \RH_\fo$.\\
(1) $\pi_\sigma \begin{pmatrix}I_n & z \\
0 & I_n\end{pmatrix}\psi(Z) = \psi(Z -z)$
expands into a convergent power series in $z_{ij}$ ($1\leqslant i\leqslant j\leqslant n$) with coefficients in $\mathscr{O}_\sigma(\Sigma(m))$;\\
(2) $\pi_\sigma \begin{pmatrix}{^th^{-1}} & 0 \\
0 & h\end{pmatrix}\psi (Z) = \sigma(h) \psi({^th}Zh)$ expands into a
convergent power series in $h_{ij}-\delta_{ij}$ ($1\leqslant i, j\leqslant n$)
with coefficients in $\mathscr{O}_\sigma(\Sigma(m))$, where
$\delta_{ij}$ is the Kronecker delta.
\end{lem}
\begin{proof}
(1) We consider the ring $\mathscr{O}(\Sigma(m))[[z]]$ of formal
power series $\varphi(z)$ in $z_{ij}$ with coefficients in
$\mathscr{O}(\Sigma(m))$; $\varphi(z)$ is expressed as
$$\varphi(z)=\sum_{\underline{r}\in \Sym(n,
\BN)}\alpha_{{\underline{r}}}\cdot
 {\underline{z}}^{{\underline{r}}} , \hskip 10pt
\alpha_{{\underline{r}}} \in \mathscr{O}(\Sigma(m)), \hskip 10pt
{\underline{z}}^{\underline{r}}:=\prod_{1 \leqslant i\leqslant j \leqslant n}
z_{ij}^{r_{ij}}.$$ If the constant term $\alpha_{\underline{0}}$ is
invertible in $\mathscr{O}(\Sigma(m))$, then $\varphi(z) \in
\mathscr{O}(\Sigma(m))[[z]]^\times$. In particular, for $(X,Y)\in
\CP$, the constant term in the expansion of $\det(X(Z-z)+Y)$, which is
$\det (XZ+Y)$, is invertible in $\mathscr{O}(\Sigma(m))$,
whence $\det(X(Z-z)+Y)^{-1} $ belongs to
$\mathscr{O}(\Sigma(m))[[z]]$.

In view of the expansion form (\ref{eq:psiZexpansion}), {the discussions in the last paragraph imply that each
coordinate of} $\psi(Z - z)$ expands into a formal power series in
$z_{ij}$ whose coefficients are series in $\mathscr{O}(\Sigma(m))$. It follows from the estimates in (\ref{eq:continuous estimate}) that 
\begin{itemize}
\item[1.] the
coefficients are indeed convergent series in
$\mathscr{O}(\Sigma(m))$ so that {each coordinate of} $\psi(Z - z)$
belongs to $\mathscr{O}(\Sigma(m))[[z]]$,
\item[2.] the
$\mathscr{O}(\Sigma(m))$-coefficient formal power series expansion
of $\psi(Z - z)$ converges in
$\mathscr{O}_\sigma(\Sigma(m))$ for all $z\in \Sym(n, \fo)$. 

\end{itemize}

(2) can be proven similarly.
\end{proof}
From (\ref{eq:decompositionofU0}), we see that $g\in
U_0(\fo)$ decomposes in $\RG_\fo$ into
$$\begin{pmatrix}I_n & z_1
\\ 0 & I_n\end{pmatrix}\begin{pmatrix} {^t}{h}{^{-1}} & 0 \\ 0 & h
\end{pmatrix} \begin{pmatrix}0 & -I_n \\
I_n & 0 \end{pmatrix} \begin{pmatrix}I_n & z_2
\\ 0 & I_n\end{pmatrix}$$ where $z_1, z_2\in \Sym(n, \fo)$ and
$h\in \RH_\fo$. Lemma \ref{lem:expansion} and (\ref{eq:continuous
estimate}) imply that $\pi_\sigma (g)\psi$ expands into a convergent $\mathscr{O}_\sigma(\Sigma(m))$-coefficient
power series whose variables are the coordinate parameters of
$U_0(\fo)$.
\end{proof}

\begin{cor}\label{cor:det-1}
The power series expansion of $\det(Z-z)^{-1}$ on $ \Sym(n, \fo)$
converges in $\mathscr{O}(\Sigma(m))$, or equivalently,
$\det(Z-z)^{-1}$ expands into a power series
$$ \sum_{{\underline{r}}\in \Sym(n, \BN)}\alpha_{{\underline{r}}}(Z)\cdot
 {\underline{z}}^{{\underline{r}}} , \hskip 10pt
\alpha_{{\underline{r}}} \in \mathscr{O}(\Sigma(m)),$$ such that
$\underset{|{\underline{r}}|\ra
\infty}{\lim}\|\alpha_{{\underline{r}}}\|_{\mathscr{O}(\Sigma(m))} =
0$, with the notation $|{\underline{r}}| = \sum_{1\leqslant i\leqslant j\leqslant n}r_{ij}$.
\end{cor}

Next, we consider the adjoint representation $\pi_\sigma^*$ of $\RG$ on
$\mathscr{O}_\sigma(\Sigma)^*_b \cong
\underset{m}{\varinjlim}\mathscr{O}_\sigma(\Sigma(m))^*_b$. The
transition homomorphisms $\mathscr{O}_\sigma(\Sigma(m))^*_b \ra
\mathscr{O}_\sigma(\Sigma)^*_b$ are injective (Proposition
\ref{prop:Stein} (2)). Lemma \ref{lem:gOmegam} implies that, for any
$g\in \RG_\fo$, $\pi_\sigma^*(g)$ maps $\mathscr{O}_\sigma
(\Sigma(m))^*_b$ into $\mathscr{O}_\sigma (\Sigma(nm))^*_b$ via
$$\langle\psi, \pi_\sigma^*(g) \mu \rangle = \langle (\pi_\sigma(g^{-1})\psi)|_{\Sigma(m)}, \mu \rangle,
\hskip 10pt \mu \in \mathscr{O}_\sigma(\Sigma(m))^*, \psi \in
\mathscr{O}_\sigma(\Sigma(nm)).$$ It is easy to deduce from
Proposition \ref{prop:pi_sigma analytic} that, for any $\mu \in
\mathscr{O}_\sigma(\Sigma(m))^*$, the orbit map
\begin{eqnarray*}  U_0(\fo)^{-1}  & \rightarrow & \mathscr{O}_\sigma(\Sigma(nm))^*_b  \\
g \quad & \mapsto & \pi_{\sigma}^*(g) \mu
\end{eqnarray*}
is an $\mathscr{O}_\sigma(\Sigma(nm))^*_b$-valued analytic function.
Therefore we have the following corollary.
\begin{cor}\label{cor:dual locally analytic}
$(\mathscr{O}_\sigma(\Sigma)^*_b, \pi_\sigma^*)$ is locally
analytic.
\end{cor}

Finally, we study the \emph{de Rham complex} $\Omega^\cdot (\Sigma)$
of rigid analytic exterior differential forms. Explicitly, let
$0\leqslant r\leqslant n(n+1)/2$,
\begin{align*}
\Omega^1_K &: = \bigoplus _{1\leqslant i\leqslant j\leq
n}K dZ_{ij}, \\
\Omega^r_K &: = \bigwedge^r \Omega^1_K(\Sigma),\\
\Omega^r(\Sigma) &:= \mathscr{O}(\Sigma)\underset{K}\otimes
\Omega^r_K .
\end{align*} As interesting examples, we show that the spaces $\Omega^r(\Sigma)$ as $\RG$-representations
are holomorphic discrete series of $\RG$ (compare
\cite{Schneider1992} \S 3).

We define a $K$-rational representation $\sigma_1$ of $\RH$ on
$\Omega^1_K$: $$\sigma_1(h) dZ_{ij} := \sum_{1\leqslant k < \ell \leq
n}(h_{ik}h_{j\ell} + h_{i\ell}h_{jk}) d Z_{k\ell} + \sum_{k=1}^n
h_{ik}h_{jk} dZ_{kk},$$ or succinctly, $$\sigma_1(h) dZ = h\cdot
dZ\cdot {^th}, \hskip 10pt dZ := \left(dZ_{ij}\right) .$$ Let
$\sigma_r := \bigwedge^r \sigma_1$.

1. For $g = \begin{pmatrix}I_n & z \\
0 & I_n\end{pmatrix}$, we have $g\cdot dZ = d (Z-z) = dZ $.

2. For $g = \begin{pmatrix}{^th^{-1}} & 0 \\
0 & h\end{pmatrix}$, we have $ g\cdot dZ = d \left(hZ\ {^th}\right) = h\cdot
dZ\cdot {^th} = \sigma_1(h) dZ$.

3. For $g = \begin{pmatrix}0 & I_n \\
-I_n & 0 \end{pmatrix}$, we have $g\cdot
dZ = d\left(- Z^{-1}\right) = Z^{-1}\cdot dZ \cdot Z^{-1} =
\sigma_1(Z^{-1}) dZ$ due to  the identity
$d\left(Z^{-1}\right)\cdot Z + Z^{-1}\cdot dZ = 0$.

In view of the decomposition (\ref{eq:decompositionofU0}) of
$U_0$, the discussions above imply that the action of
$\RG$ on $\Omega^r(\Sigma)$ coincides with $\pi_{\sigma_r}$ on
$U_0$ and hence on   $\RG$ as $U_0$ is dense in $\RG$. 

\begin{prop}
Let $1\leqslant r\leqslant n(n+1)/2$ and $\sigma_r$ be defined above. The
$\RG$-action on $\Omega^r(\Sigma)$ coincides with $ \pi_{\sigma_r}$.
\end{prop}

\section{Duality}\label{sec:Duality}

In the following, we  assume that $K$ is spherically complete. Let
$(V, \sigma)$ be a $d$-dimensional $K$-rational representation of
$\RH$. We choose a basis $v_1, \cdots, v_d$ of $V$; we denote by
$v_1^*, \cdots, v_d^*$ the corresponding dual basis of the dual
space $V^*$. Let $(V^*, \sigma^*)$ denote the dual representation of
$(V, \sigma)$.

\subsection{The duality operator $I_{\sigma}$} \label{sec:Isigma}

For $Z\in \Sigma$ and $v^*\in V^*$, let $\varphi_{Z, v^*}$ be the
$V^*$-valued locally analytic function on $\CP$:
\begin{equation}\label{eq:varphiZv*}\varphi_{Z, v^*}(X,
Y):=\sigma^*(XZ+Y)v^*.\end{equation}

Let $B^0_{\sigma^*}(\CP, V^*)$ be the subspace of
$C^\an_{\sigma^*}(\CP, V^*)$ spanned by $\varphi_{Z, v^*}$,
$B_{\sigma^*}(\CP, V^*)$ the closure of $B^0_{\sigma^*}(\CP, V^*)$.
Clearly $B _{\sigma^*}(\CP, V^*)$ is $\RG$-invariant.

For any continuous linear functional $\xi\in B _{\sigma^*}(\CP,
V^*)^* $, we define a $V$-valued function on $\Sigma$:
\begin{equation} \label{eq:Isigma}
I_\sigma(\xi) (Z) := \sum_{k=1}^d \langle \varphi_{Z, v^*_k}, \xi
\rangle v_k, \hskip 10pt Z\in \Sigma.
\end{equation}
One verifies that $I_\sigma(\xi)$ is independent of the choice of the
basis  $\{v_k \}_{k=1}^d$. Evidently, $I_\sigma$ is injective.

\begin{lem} \label{lem:Iequivariant}
$I_\sigma$ is $\RG$-equivariant, that is,
$$I_\sigma(T_{\sigma^*}^*(g) \xi) = \pi_\sigma (g) I_\sigma (\xi),$$
for any $g\in \RG$.
\end{lem}
\begin{proof}
Let $g = \begin{pmatrix}A &B \\ C & D\end{pmatrix} \in \RG$. We have
\begin{align*}  I_\sigma(T_{\sigma^*}^*(g) \xi)(Z) 
= \ & \sum_{k=1}^d  \langle \varphi_{Z, v_k^*}, T_{\sigma^*}^*(g) \xi \rangle v_k  \\
= \ & \sum_{k=1}^d  \langle T_{\sigma^*}(g^{-1})\varphi_{Z, v_k^*}, \xi \rangle v_k \\
=\ & \sum_{k=1}^d  \langle \sigma^*(X({^tD}Z - {^tB})+Y (-{^tC}Z + {^tA}))v^*_k, \xi \rangle v_k \\
=\ & \sigma (j(g^{-1}, Z))^{-1} \Big(\sum_{k=1}^d  \langle
\sigma^*(X \cdot g^{-1}Z + Y) (v_{k; g})^*, \xi \rangle v_{k; g}\Big) \\
=\ & (\pi_\sigma (g) I_\sigma (\xi))(Z),
\end{align*}
where $v_{k; g} = \sigma (j(g^{-1}, Z))v_k$.
\end{proof}

\begin{prop}\label{prop:Isigma}$\ $ 

(1) For any continuous linear functional $\xi\in B _{\sigma^*}(\CP,
V^*)^* $, $I_\sigma(\xi)$ is a $V$-valued rigid analytic function on
$\Sigma$. 

(2) $ I_\sigma $ is a continuous homomorphism of
$\RG$-representations from $(B _{\sigma^*}(\CP, V^*)^*_b,
T^*_{\sigma^*})$ to $(\mathscr{O}_\sigma(\Sigma), \pi _{\sigma})$.
\end{prop}
\begin{proof} Let  $i$ denote  the inclusion: $B _{\sigma^*}(\CP, V^*)
\hookrightarrow C^\an_{\sigma^*}(\CP , V^*)$, and $i^*$ be its adjoint
operator. Because of our assumption that $K$ is spherically
complete, the Hahn-Banach Theorem (\cite {Schneider} Corollary 9.4)
implies that $i^*$ is surjective. Since $C^\an_{\sigma^*}(\CP ,
V^*)^*_b$ and $B _{\sigma^*}(\CP, V^*)^*_b$ are Fr\'echet spaces
(Corollary \ref{cor:compacttype}), $i^*$ is open (from the open
mapping theorem (\cite{Schneider} Proposition 8.6)). Consequently, the
continuity of $I_\sigma\circ i^*$ implies that of $I_\sigma$.
Therefore,
(1) and (2) are equivalent to:\\
($1'$) $I_\sigma\circ i^*(\xi)\in \mathscr{O}_\sigma(\Sigma)$ for
any $\xi\in C^\an_{\sigma^*}(\CP, V^*)^* $;
\\
($2'$) $I_\sigma\circ i^*: (C^\an_{\sigma^*}(\CP , V^*)^*_b,
T^*_{\sigma^*}) \ra (\mathscr{O}_\sigma(\Sigma), \pi _{\sigma})$ is
a continuous homomorphism of $\RG$-representations.
\\For brevity, we  still denote $I_\sigma\circ i^*$ by
$I_\sigma$. For $\xi\in C^\an_{\sigma^*}(\CP, V^*)^* $, we write
$I_\sigma(\xi)$ in the form of integral:
\begin{align*}I_{\sigma}(\xi)(Z)&=   \sum_{k=1}^d \int_{\CP} \varphi_{Z; v^*_k} d \xi\cdot v_k \\
& =   \sum_{k=1}^d \sum_{\kappa } \int_{\CU_{\kappa }}
\varphi_{Z; v^*_k} d \xi\cdot v_k\\
& =  \sum_{\kappa } \pi_{\sigma}(g_{\kappa }) \Big( \sum_{k=1}^d
\int_{\CU_{\kappa }\cdot g_{\kappa }} \varphi_{Z; (v_{k; g_\kappa
})^*} d (T_{\sigma^*}(g_{\kappa }^{-1})\xi)\cdot v_{k; g_\kappa
}\Big),
\end{align*}
where the disjoint open covering $\{\CU_\kappa \}_\kappa$ of $\CP$ and $g_\kappa$ are defined
in \S \ref{sec:principal}, and  $v_{k; g_\kappa }$ is defined in the proof of Lemma
\ref{lem:Iequivariant}. Therefore it suffices to consider
\begin{equation}\label{eq:integerU}
\sum_{k=1}^d \int_{\CU} \varphi_{Z; v^*_k} d \xi'\cdot v_k.
\end{equation}
where $\CU$ are taken to be $\CU_{\kappa }\cdot g_{\kappa } $ and
$\xi'$ is the image of $\xi$ under $C^\an_{\sigma^*}(\CP, V^*)^*_b
\ra C^\an_{\sigma^*}(\CU, V^*)^*_b$.

For the open subset $\overline{\CU} = \mathrm{pr}^\CP_\CL (\CU)$ of
$\Sym(n, \fo)$, we have the isomorphism induced from the section
$\iota_0$ (compare (\ref{eq:iotacirc})):
\begin{equation}\label{eq:isoCO} C^\an_{\sigma^*}(\CU , V^*)^*_b\cong C^\an(\overline{\CU}, V^*)^*_b.
\end{equation}
Then (\ref{eq:integerU}) is equal to $$\overline{I}_{\sigma,
\overline{\CU}}(\overline{\xi})(Z) := \sum_{k=1}^d
\int_{\overline{\CU}} (\sigma^*(Z-z)v^*_k) d \overline{\xi}(z)\cdot
v_k,$$ where $\overline{\xi}$ is the image of $\xi'$ in
$C^\an(\overline{\CU}, V^*)^*_b$ via the isomorphism
(\ref{eq:isoCO}). Therefore, it suffices to prove that $\overline{I}_{\sigma,
\overline{\CU}}(\overline{\xi})$ is rigid analytic on $\Sigma(m)$,
and that the map
\begin{eqnarray*}C^\an(\overline{\CU}, V^*)^*_b &\ra& \mathscr{O}_\sigma(\Sigma(m))\\
\overline{\xi} &\mapsto& \overline{I}_{\sigma,
\overline{\CU}}(\overline{\xi})|_{\Sigma(m)}
\end{eqnarray*}
is continuous ($\RG$-equivariance is already proven in Lemma
\ref{lem:Iequivariant}).

As $\sigma^*$ is algebraic, there is a nonnegative integer $t$ and
polynomials $Q_{k\ell}$ $(1\leqslant k,\ell\leqslant d)$ in $h_{ij}$ $(1\leq
i, j\leqslant n)$ with coefficients in $K$, such that
$$\sigma^*(h)v_k^*=\sum_{\ell=1}^d\det(h)^{-t}Q_{k\ell}(h)v^*_\ell.$$
We expand
$$\det(Z-z)^{-t}Q_{k\ell}(Z-z)=
\sum_{{\underline{r}}}\alpha_{{\underline{r}}, k \ell} (Z)\cdot
{\underline{z}}^{{\underline{r}}}.$$
It is evident from Corollary \ref{cor:det-1} that $\alpha_{{\underline{r}},
k \ell}\in \mathscr{O}(\Sigma(m))$ and
\begin{equation}\label{eq:bnkl}\underset{|{\underline{r}}|\ra
\infty}{\lim}\|\alpha_{{\underline{r}}, k
\ell}\|_{\mathscr{O}(\Sigma(m))} = 0.
\end{equation}
Moreover, there is a constant $c_m>0$, depending only on $m$,
$\sigma$ and $\{v_k\}_{k=1}^d$, such that
\begin{equation}\label{eq:bnkl2} \|\alpha_{{\underline{r}}, k \ell}\|_{\mathscr{O}(\Sigma(m))}\leqslant c_m .
\end{equation}
Then
\begin{align}\nonumber \overline{I}_{\sigma, \overline{\CU}}(\overline{\xi})(Z) &= \sum_{k,\
\ell=1}^d \int_{\overline{\CU}} \det(Z-z)^{-t}Q_{k\ell}(Z-z) d
\overline{\xi}(z)\cdot
v_k\\
&= \label{eq:expansionbarI}\sum_{k=1}^d \Big(\sum_{\ell=1}^d
\sum_{{\underline{r}}}\Big(\int_{\overline{\CU}}
 {\underline{z}}^{{\underline{r}}}\cdot v^*_\ell d\overline{\xi}(z)\Big)\cdot \alpha_{{\underline{r}}, k\ell } (Z) \Big) v_k.
\end{align}
Since $\|{\underline{z}}^{{\underline{r}}}\|_{C^\an(\overline{\CU})}
\leqslant 1$, we have
\begin{equation}\label{eq:intznvl} \left|\int_{\overline{\CU}}
{\underline{z}}^{{\underline{r}}}\cdot v^*_\ell
d\overline{\xi}(z)\right|\leqslant \|v^*_\ell\|_{V^*}\cdot
\|\overline{\xi}\|_{C^\an(\overline{\CU}, V^*)^*_b}.\end{equation}
In conclusion, (\ref{eq:bnkl}) and (\ref{eq:intznvl}) imply that the expansion
(\ref{eq:expansionbarI}) of $\overline{I}_{\sigma,
\overline{\CU}}(\overline{\xi})$ converges in
$\mathscr{O}_\sigma(\Sigma(m))$, whereas (\ref{eq:bnkl2}) and
(\ref{eq:intznvl}) imply
$$\left\|\overline{I}_{\sigma, \overline{\CU}}(\overline{\xi}) \right\|_{\mathscr{O}_\sigma(\Sigma(m))}\leq
\max_{1\leqslant k, \ell\leqslant d} c_m \|v_\ell^*\|_{V^*} \|v_k\|_V\cdot
\|\overline{\xi}\|_{C^\an(\overline{\CU}, V^*)^*_b}.
$$ The continuity follows.
\end{proof}

\subsection{The duality operator $J_\sigma$ and the image of $I_\sigma$}

Let $\mathscr{N}_\sigma (\Sigma)$ denote the image of $I_\sigma$. In
this section, we propose to determine $\mathscr{N}_\sigma (\Sigma)$.
For this, we  introduce $J_\sigma$, the adjoint operator of
$I_\sigma$:
an injective continuous linear operator from 
$\mathscr{N}_\sigma(\Sigma)^*_b$ to $(B_{\sigma^*}(\CP,
V^*)^*_b)_b^{*}\cong B_{\sigma^*}(\CP, V^*)$ ($B_{\sigma^*}(\CP,
V^*)$ is reflexive according to Corollary \ref{cor:compacttype}).

First, we need to find the formula for $J_\sigma$.

For any $\mu\in \mathscr{N}_\sigma(\Sigma)^*$ and $\xi\in
B_{\sigma^*}(\CP, V^*) ^*$, we have 
\begin{equation} \label{eq:dual} \langle J_\sigma(\mu), \xi\rangle=\langle I_\sigma(\xi), \mu\rangle.
\end{equation}

For $(X, Y)\in \CP$ and $v\in V$, we define the Dirac distribution
$\xi_{(X,Y), v}$, which is a continuous linear functional
of $ B_{\sigma^*}(\CP, V^*)$,  as follows:
$$\langle \varphi, \xi_{(X,Y), v}\rangle = \langle v,
\varphi(X,Y)\rangle_V, \hskip 10pt \varphi \in B_{\sigma^*}(\CP,
V^*), $$ and a $V$-valued rigid analytic function $\psi_{(X,Y), v}$
on $\Sigma$:
\begin{equation}\label{eq:psiXYv}\psi_{(X,Y), v} (Z) :=
\sigma(XZ+Y)^{-1}v.\end{equation}

\begin{lem}\label{lem:Jsigma-xi} $$I_\sigma(\xi_{(X,Y),
v})=\psi_{(X,Y), v} . $$
\end{lem}
\begin{proof} 
	 By definition (\ref{eq:Isigma}), {\allowdisplaybreaks
	 	\begin{align*}
	 		\big(I_\sigma(\xi_{(X,Y), v})\big)(Z) &= \sum_{k=1}^r\langle
	 		\varphi_{Z, v_k^*}, \xi_{(X,Y),
	 			v}\rangle v_k  \\
	 		&= \sum_{k=1}^r \langle v, \sigma^*(XZ+Y)v_k^* \rangle_V \cdot v_k \\
	 		&= \sum_{k=1}^r \langle \sigma(XZ+Y)^{-1} v, v_k^*\rangle_V \cdot v_k  \\
	 		&= \sigma(XZ+Y)^{-1} v = \psi_{(X, Y),v}(Z).
	 	\end{align*}
	 }
\end{proof}

Let $\mathscr{N}_\sigma^0(\Sigma)$ denote the subspace of
$\mathscr{O}_\sigma(\Sigma)$ spanned by $\psi_{(X,Y), v}$ for all
$(X,Y)\in \CP $ and $v\in V$. Clearly $\mathscr{N}^0_\sigma(\Sigma)$
is $\RG$-invariant. Lemma \ref{lem:Jsigma-xi} implies
$\mathscr{N}_\sigma^0(\Sigma)\subset \mathscr{N}_\sigma(\Sigma)$.

\begin{prop} For any continuous linear
functional $\mu\in \mathscr{N}_\sigma(\Sigma)^*$, we have
\begin{equation}\label{eq:Jsigma}
J_{\sigma}(\mu)(X, Y)=\sum_{k=1}^d \langle\psi_{(X,Y), v_k},
\mu\rangle v_k^*.
\end{equation}
\end{prop}
\begin{proof} We have
	{\allowdisplaybreaks
		\begin{align*} \sum_{k=1}^r  \langle\psi_{(X,Y), v_k},
			\mu\rangle v_k^* &= \sum_{k=1}^r \langle I_\sigma(\xi_{(X,Y),
				v_k}), \mu
			\rangle v_k^*\hskip 30pt (\text{Lemma \ref{lem:Jsigma-xi}})\\
			&= \sum_{k=1}^r\langle J_\sigma(\mu), \xi_{(X,Y), v_k}\rangle v_k^*
			\hskip 30pt
			(\text{Duality formula (\ref{eq:dual})}) \\
			&=  \sum_{k=1}^r \langle v_k, J_\sigma(\mu)(X, Y) \rangle_V \cdot
			v_k^* =
			J_\sigma(\mu)(X, Y).
		\end{align*} }
\end{proof}

It follows from (\ref{eq:Jsigma}) that $J_\sigma$ factors through
$\mathscr{N}_\sigma^0(\Sigma)^*$ and (\ref{eq:Jsigma}) defines an
injection from $\mathscr{N}^0_\sigma(\Sigma)^*_b $ to $
B_{\sigma^*}(\CP, V^*)$. Because $J_\sigma$ is injective and
$\mathscr{N}_\sigma(\Sigma)^*_b\ra \mathscr{N}^0_\sigma(\Sigma)^*_b$
is surjective (the Hahn-Banach Theorem), we have
$\mathscr{N}^0_\sigma(\Sigma)^*_b = \mathscr{N}_\sigma(\Sigma)^*_b$.
The following lemma then follows from the Hahn-Banach theorem.

\begin{lem}\label{lem:CN0denseinCN}$\mathscr{N}^0_\sigma(\Sigma)$ is dense in $\mathscr{N}_\sigma(\Sigma)$.
\end{lem}

\begin{thm} \label{thm:Iisomorphism}~ 
	
(1) $I_\sigma $ is an isomorphism from $B _{\sigma^*}(\CP, V^*)^*_b$
onto $\mathscr{N}_\sigma(\Sigma)$.

(2) $\mathscr{N}_\sigma(\Sigma)$ is the closure of
$\mathscr{N}^0_\sigma(\Sigma)$ in $\mathscr{O}_\sigma(\Sigma)$.
\end{thm}
\begin{proof}

Let $B(\CL, V^*) := \iota^{\circ}(B_{\sigma^*}(\CP, V^*))$. We 
still denote $\iota^{\circ}|_{B_{\sigma^*}(\CP, V^*)} $ by $\iota
^{\circ}$.

Let $\CI$ be any (finite) disjoint open chart covering
$\{\overline{\CU}_i\}_{i}$ of $\CL$. We recall that $C^\an(\CL,
V^*)$ is defined as the inductive limit of the $K$-Banach algebra $E_{\CI}(\CL, V^*) = \prod_{i} \CO
(\overline{\CU}_i, V^*)$, indexed with all the $\CI$,
where $\CO (\overline{\CU}_i, V^*)$
denotes the space of $K$-analytic functions on $\overline{\CU}_i$
(cf. \cite{Feaux} 2.1.10 and \cite{Schneider-Teitelbaum2002} \S 2).
The inductive limit structure is naturally induced onto $B(\CL,
V^*)$, that is, $B(\CL, V^*) = {\varinjlim}_{\CI} E_{\CI}(\CL,
V^*).$ Moreover, the dual space $B(\CL, V^*)^*_b$ is the projective
limit of $E_{\CI}(\CL, V^*)^*_b.$

Let $\mathscr{N}^0_\sigma(\Sigma(m))$ be the image of
$\mathscr{N}^0_\sigma (\Sigma )$ in $\mathscr{O}_\sigma(\Sigma(m))$.

Considering $\pi_\sigma(g^{-1}) v_k$, we see that the map $(X, Y)
\mapsto \psi_{(X,Y), v_k}$ is an
$\mathscr{O}_\sigma(\Sigma(m))$-valued locally analytic map on $\CP$
(see Proposition \ref{prop:pi_sigma analytic}). Define $$r_m = \underset {1\leqslant k\leqslant d} {\min}\underset {(X,Y)\in
\CK} {\inf} \|\psi_{(X,Y), v_k}\|_{\mathscr{O}_\sigma(\Sigma(m))}$$
Since $\CK $ is
compact, $r_m$ is positive. Let $\mathscr{L}$ be the lattice $ \sum_{k=1}^d
\underset{(X,Y)\in \CK}{\sum} \fo_K\cdot \psi_{(X,Y), v_k} $ in
$\mathscr{N}^0_\sigma(\Sigma)$. For each $m$, the image of
$\mathscr{L}$ in $\mathscr{N}^0_\sigma(\Sigma(m))$ contains the ball
of radius $r_m$ centered at zero, and therefore the interior of
$\mathscr{L}$ is a nontrivial open lattice.

Consider
\begin{eqnarray*} ({\iota ^{\circ}}^{-1})^*  \circ I_\sigma^{-1} |_{\mathscr{N}^0_\sigma(\Sigma )}: \mathscr{N}^0_\sigma(\Sigma )& \ra & B (\CL, V^*)^*_b\\
\psi_{(X, Y), v} &\mapsto& ({\iota ^{\circ}}^{-1})^* (\xi_{(X,Y),
v}),
\end{eqnarray*}

For $(X,Y)\in \CK$,
\begin{align*}\|({\iota
^{\circ}}^{-1})^* (\xi_{(X,Y), v})\|_{E_{\CI}(\CL, V^*)^*_b}& =
\max _{\overline{\varphi} \in\ E_{\CI}(\CL, V^*)}\frac{\langle
\overline{\varphi}, ({\iota
^{\circ}}^{-1})^* (\xi_{(X,Y), v})\rangle }{\|\overline{\varphi}\|_{E_{\CI}(\CL, V^*)}}\\
& = \max _{\varphi\in\ {\iota ^{\circ}}^{-1}(E_{\CI}(\CL, V^*))}
\frac{\langle {\varphi}, \xi_{(X,Y), v} \rangle }{\|\iota ^{\circ}(\varphi)\|_{E_{\CI}(\CL, V^*)}}\\
&= \underset{\varphi\in\ {\iota ^{\circ}}^{-1}(E_{\CI}(\CL, V^*)) }
{\max}\frac{\langle v, \varphi(X,Y)\rangle_V}{\underset{(X', Y')\in
\CK} {\max}\|\varphi(X',
Y')\|_{V^*}}\\
&\leqslant \|v\|_V.
\end{align*}

Therefore the image of $\mathscr{L}$ under $({\iota
^{\circ}}^{-1})^*  \circ I_\sigma^{-1}
|_{\mathscr{N}^0_\sigma(\Sigma )}$ in  $B (\CL, V^*)^*_b$ is
bounded, since its image in $E_{\CI} (\CL, V^*)^*_b$ are all
norm-bounded by $\underset{1\leqslant k\leqslant d}{\max}\|v_k\|_V$. Because
$\mathscr{N}^0_\sigma(\Sigma)$ is metrizable, it is bornological
(\cite{Schneider} Proposition 6.14), and therefore
$I_\sigma^{-1}|_{\mathscr{N}^0_\sigma(\Sigma )}$ is continuous
(\cite{Schneider} Proposition 6.13). $\mathscr{N}^0_\sigma(\Sigma)$
is isomorphic to $I_\sigma^{-1}(\mathscr{N}^0_\sigma(\Sigma))$, then
their completions are isomorphic, which, in view of Lemma
\ref{lem:CN0denseinCN}, must be $\mathscr{N}_\sigma(\Sigma)$ and $B
_{\sigma^*}(\CP, V^*)^*_b$, respectively.
\end{proof}




\delete{ By (\ref{eq:Jsigma}), {\allowdisplaybreaks
\begin{eqnarray*}
\big(I_\sigma(\m\CU_{Z,v^*})\big)(X,Y) &=&
\sum_{k=1}^r\langle \psi_{(X,Y), v_k}, \m\CU_{Z,v^*}\rangle v_k^*  \\
&=& \sum_{k=1}^r \langle\sigma(XZ+Y)^{-1}v_k, v^*\rangle_V \cdot v_k^* \\
&=& \sum_{k=1}^r \langle v_k, \sigma^*(XZ+Y) v^*\rangle_V \cdot v_k^*  \\
&=& \sigma^*(XZ+Y) v^* \\
&=& \varphi_{Z,v^*}(X, Y).
\end{eqnarray*}
} }


\begin{cor}
$J_\sigma$ is an isomorphism of $\RG$-representations from
$(\mathscr{N}_\sigma(\Sigma)^*_b, \pi_{\sigma}^*)$ onto
$(B_{\sigma^*}(\CP, V^*), T_{\sigma^*})$.
\end{cor}

\begin{rem}\label{rem:irred}
We conjecture that $(\mathscr{N}_\sigma(\Sigma), \pi_\sigma)$ and
$(B_{\sigma^*}(\CP, V^*) , T_{\sigma^*})$ are topologically
irreducible $\RG$-representations if $\sigma$ is irreducible. These
are conjectured and claimed by Morita for $\SL(2, F)$
(\cite{MoritaII} Corollary after Theorem 3 and \cite{MoritaIII}
Theorem 1 (i).) However, there is a serious gap in his proof of
\cite{MoritaIII} Proposition 3. Schneider and Teitelbaum gave the
first valid proof of \cite{MoritaIII} Theorem 1 (i) in
\cite{Schneider-Teitelbaum2002} when $F=\BZ_p$.
\end{rem}

\section{Morita's theory for $\SL(2,F)$}\label{sec:SL2Morita}
In this section, we study Morita's theory for $ \Sp(2,F) = \SL(2,F)
$. 
We start with reviewing the constructions of   holomorphic discrete series
and principal series for $\SL(2,F)$ from \cite{MoritaI},
\cite{MoritaII} and \cite{MoritaIII} in accordance with our
notations. Then we focus on the duality established in \textsection
\ref{sec:Duality} for $\SL(2,F)$ and   its relation with
Morita's duality and Casselman's intertwining operator.

\subsection{The $p$-adic upper half-plane} \label{sec:upperhalfplane}

{For more details, we refer the readers to \cite{MoritaI} \textsection 2 and \cite{DasTei:Lecture} 1.2.}

In the following, let $\RG=\SL(2, F)$ and $\RG_\fo = \SL(2, \fo)$.

Let $\Sigma := K -  F$ be the $p$-adic upper half-plane,
$\CP$ be the set of nonzero pairs $(x,y)\in F\times F$, {$\CL = F^{\times}\backslash \CP =
\BP^1(F)$}, $\CP_\fo\subset \CP$ be the set
of pairs $(x,y)\in \fo\times \fo$ such that $(x,y) \nequiv (0, 0)
\mod \varpi $. As usual, we define a $\RG$-action on $\Sigma$ by
$$g\cdot Z := (aZ+b)(cZ+d)^{-1}, \hskip 10pt g=\begin{pmatrix}a & b \\
c & d\end{pmatrix}\in \RG.$$

Let $m$ be a nonnegative integer. For a pair $(x, y)\in\CP_\fo $,
we define
$$
B^-(m; x,y) := \left\{ Z\in K \ |\ |xZ+y| < \max \{1, |Z|\} \
|\varpi^{m}| \right\}.$$

Let \begin{align*} \Sigma (m)& :=  \bigcap_{(x, y)\in\CP_\fo} K -  B^- (m; x, y)\\
&=  \left\{ Z \in \Sigma \ |\ | xZ+y | \geqslant  \max \{1, |Z|\}\
|\varpi^{ m}| \text{ for any } (x, y)\in\CP_\fo  \right\}.
\end{align*}

{It is not hard to verify that the 
admissible affinoid covering $\{\Sigma(m)\}_{m=0}^\infty$ of $\Sigma$ coincide with that defined
in  \cite{DasTei:Lecture} 1.2.}

Let $\mathscr{O}(\Sigma(m))$ be the space of $K$-valued rigid
analytic functions on $\Sigma(m)$. Explicitly, on taking partial
fractional expansion of each summand in (\ref{eq:psiZexpansion}), one sees
  that $\psi\in \mathscr{O}(\Sigma(m))$ is a $K$-valued functions
on $\Sigma(m)$ that has an expansion in the form:
$$\psi (Z) = \sum_{i=0}^{\infty} a_i^{(\infty)} Z^i + \sum _{j=1}^{\ell} \sum_{i=-1}^{-\infty} a_i^{(j)}(Z-z_j)^i,$$
where $\ell \geqslant 0$, $a_i^{(\infty)}, a_i^{(j)}\in K$, $z_j\in F$, and the expansion
converges with respect to the supremum norm. The space of $K$-rigid
analytic functions on $\Sigma$ is the projective limit of
$\mathscr{O} (\Sigma(m))$.

\subsection{Holomorphic discrete series of $\SL(2,F)$}

Let $s$ be an integer. We define the holomorphic discrete series
$(\mathscr{O}(\Sigma), \pi_s)$ of $\RG$ (see (\ref{eq:discrete});
compare \cite{MoritaI} \textsection 3-1.):
\begin{equation}\label{eq:discrete series}
\pi_s(g)\psi(Z) := (-cZ+a)^{-s}\psi((dZ-b)(-cZ+a)^{-1}),
\end{equation}
with $g= \begin{pmatrix}a &b \\ c & d\end{pmatrix}\in \RG$ and
$\psi \in \mathscr{O}(\Sigma)$. $\pi_s$ is a continuous
representation of $\RG$.

Let $\mathscr{N}^0_s(\Sigma)$ be the subspace of
$\mathscr{O}(\Sigma)$ spanned by $1$ and $\psi^{(s)}_{z}(Z) :=
(Z-z)^{-s}$  for all $z \in F$ (see (\ref{eq:psiXYv})), and let
$\mathscr{N}_s(\Sigma)$ be the closure of $\mathscr{N}^0_s(\Sigma)$.
$\mathscr{N}_s(\Sigma)$ is $\RG$-invariant.

If $s\leqslant 0$, $\mathscr{N}_s(\Sigma)$ is obviously the space of
polynomial functions  of degree $\leqslant -s$.

If $s > 0$, let $\tilde{\mathscr{N}}^0_s(\Sigma)$
be the subspace of $\mathscr{O}(\Sigma)$ consisting of all rational
functions $\psi$ that has a partial fractional expansion of the
form
$$\psi (Z) = \sum_{i=0}^{\infty} a_i^{(\infty)} Z^i + \sum _{j=1}^{\ell} \sum_{i=-s}^{-\infty} a_i^{(j)}(Z-z_j)^i,$$
where the sum is finite, with $\ell \geqslant 0$, $z_j\in F$ and
$a_i^{(\infty)}, a_i^{(j)}\in K$. One verifies that $\tilde{\mathscr{N}}^0_s(\Sigma)$ is
$\RG$-invariant. Moreover, let $\tilde{\mathscr{N}}_s(\Sigma)$ be the closure
of $\tilde{\mathscr{N}}^0_s(\Sigma)$ in $\mathscr{O}(\Sigma)$.

The next lemma follows immediately from \cite{MoritaI} Theorem 2
(i).

\begin{lem}\label{lem:tildeCN}
Let $s$ be a positive integer. $\tilde{\mathscr{N}}_s(\Sigma)$ is the smallest $\RG$-invariant closed
subspace of $\mathscr{O}(\Sigma)$ containing $1$.
\end{lem}

We note that $ 1 \in \mathscr{N}_s(\Sigma)$ and
$\mathscr{N}^0_s(\Sigma)\subset \tilde{\mathscr{N}}^0_s(\Sigma)$,
and therefore we have the following proposition.

\begin{prop}
Let $s$ be a positive integer. $\mathscr{N}_s(\Sigma) =
\tilde{\mathscr{N}}_s(\Sigma)$.
\end{prop}

\subsection{Principal series of $\SL(2,F)$}
\label{sec:SLPrincipalSeries}

{The references for this section are \cite{MoritaII} \S 2, 3 and \cite{MoritaIII} \S 2.}

Let $s$ be an integer. Define  the character of $F^{\times}$,
${\chi_s}(z)=z^s$. Let $C^{\an}_{\chi_s}(\CP)$ denote the space of $K$-valued
locally analytic functions $\varphi$ on $\CP$ satisfying
$$\varphi(hx, hy) = {\chi_s}(h)\varphi(x,y), \hskip 10pt (x,y)\in \CP, h\in F^{\times}.$$
In the following, we identify $(1, F)$ with $F$ via $(1, -z) \ra z$
and write $\varphi(z) = \varphi(1,-z)$ and $\varphi(\infty) =
\varphi(0,1)$. Then $\varphi(z)$ is a locally analytic function on
$F$ that has Laurent expansion at infinity of the form:
$$\varphi(z) = \sum_{i=s}^{-\infty} b_i^{(\infty)} z^i,
\hskip 10pt b_i^{(\infty)}\in K.$$ Clearly, $\varphi(\infty) =
(-1)^s b_s^{(\infty)}$.
Let $D_s$ denote the space of all such functions
$\varphi(z)$ on $F$. We have a $K$-linear bijective map between $D_s $ and $
C_{\chi_s}^{\an}(\CP)$; we endow $D_s$ with the topology that makes this map
into an isomorphism. Then the representation $(C^{\an}_{\chi_s}(\CP), T_{{\chi_s}})$ of $\RG$, 
 defined by (\ref{eq:defTsigma}), can be
realized as the representation $(D_s, T_s)$:
\begin{equation}\label{eq:defTchi}T_s(g) \varphi (z)
:= (-cz+a)^s\varphi((dz-b)(-cz+a)^{-1}), \hskip 10pt
g=\begin{pmatrix}a & b
\\ c & d\end{pmatrix}\in \RG,  \varphi \in D_s.\end{equation}

Let $B^0_s$ be the subspace of $D_s$ spanned by the $K$-valued
locally analytic functions $\varphi^{(s)}_{Z}(z) := (Z-z)^s$ for all $Z\in \Sigma$ (see
(\ref{eq:varphiZv*})), and let $B_s$ be the closure
of $B^0_s$. $B _s$ is $\RG$-invariant.

\delete{

If $s$ is a negative integer, Morita proved
\begin{lem}[cf. \cite{MoritaIII} Theorem 1 (i) and \cite{Schneider-Teitelbaum2002} Theorem
6.1.]\label{lem:Dmirreducible} Let $s$ be a negative integer, then
$(D_s, T_s)$ is a topologically irreducible $\RG$-representation.
\end{lem}

By Lemma \ref{lem:Dmirreducible}, we obtain

\begin{prop}\label{prop:Bs=Ds}
If $s$ is  a negative integer, then $B_s=D_s$.
\end{prop}

\begin{rem}
In the proof of \cite{MoritaIII} Proposition 3, the reduction to the
case where $K$ is algebraically closed fails to hold. We must assume
that $K$ is algebraically closed in \cite{MoritaIII} Theorem 1.
However, Schneider and Teitelbaum proved \cite{MoritaIII} Theorem 1
(i) in \cite{Schneider-Teitelbaum2002} as an application of locally
analytic distributions.
\end{rem}
}

If $s$ is a nonnegative integer, let $P_s^{\loc}$ denote the subspace of
$D_s$ consisting of $K$-valued functions $\varphi(z)$ on $F$ such
that the local Taylor expansion at each point of $F$ and the Laurent
expansion at infinity of $\varphi(z)$ are both given by polynomials
of degree $\leqslant s$. $P_s^{\loc}$ is $\RG$-invariant. Observe that $\varphi\in
P_s^{\loc}$ if and only if its $(s+1)$-th derivation
$(d/dz)^{s+1}\varphi (z) \equiv  0$. Let $P_s$ denote the space of all
polynomial functions on $F$ of degree $\leqslant s$. $P_s^{\loc}$ and
$P_s$ are both closed $\RG$-invariant subspaces. Clearly $B_s =
P_s$.

The subspace of $C^{\an}_{\chi_s} (\CP)$ corresponding to $P_s$ (resp.
$P_s^\loc$) is the space of (resp. locally) homogeneous polynomial
functions $\varphi(x, y) $ on $\CP$ of degree $s$.

In addition, we define $P^\loc_{-1}= P_{-1} = 0$.

\begin{prop}[Casselman's intertwining
operator]\label{prop:Casselman}Let $s\geqslant -1$. The $s+1$-th differentiation map
\begin{eqnarray}\notag \label{eq:defCasselman}
S_s: D_s &\ra& D_{-s-2}\\
\varphi(z)&\mapsto& (d/dz)^{s+1}\varphi (z)
\end{eqnarray}
induces a  $\RG$-isomorphism from $D_s/P^\loc_s$ onto $D_{-s-2}$.
\end{prop}

\delete{
\begin{proof}
In view of the Laurent expansion of $\varphi(z)$ at infinity, it is
easy to see that $\varphi^{(s+1)} (z)$ belongs to $D_{-s-2}$.
Evidently, $S_s$ is a surjective continuous linear operator with
kernel $P^\loc_s$.

Next, we prove the $\RG$-equivariance. Since $\RG$ is generated by
$R(\delta)$ ($\delta \in F$) and $\mathfrak{w}$, despite the obvious
case, it suffices to prove
$$
\left(\frac {d}{dz}\right)^{s+1}\big(z^s\varphi(-1 /z)\big) =
z^{-s-2}\varphi^{(s+1)} (- 1 /z).
$$
This equality follows from
$$\left(\frac {d}{dz}\right)^{i}\big(z^s\varphi(-1 /z)\big) =
\sum_{j=0}^i \left(\begin{matrix}i\ \\
j\ \end{matrix}\right) \cdot \prod_{k=j}^{i-1}(s-k) \cdot
z^{s-i-j}\varphi^{(j)} (- 1 /z),$$ which can be easily proven by
induction.

(\ref{eq:defCasselman}) induces a $\RG$-equivariant continuous
linear bijection $S_s$ from $D_s/P^\loc_s$ onto $D_{-s-2}$. The dual
operator $S_s^* : (D_{-s-2})^*_b \ra (D_s/P^\loc_s)^*_b$ is a
continuous linear bijection between two Fr\'echet spaces (Corollary
\ref{cor:compacttype}). By Corollary 8.7 to the open mapping theorem
in \cite{Schneider}, $S^*_s$ is an isomorphism, then so is $S_s$,
because both spaces are reflexive (Corollary \ref{cor:compacttype}).
\end{proof}
}


\subsection{Morita's duality for $\SL(2, F)$}

\begin{defn}[cf. \cite{MoritaII} \textsection 5.]
Let $s$ be an integer. 

1. We call the following $K$-linear pairing
$\langle\ , \ \rangle_M ^{(s)} : D_{s-2} \times \mathscr{O}(\Sigma)
\ra K$ \textbf{Morita's pairing}:
\begin{equation}
\langle  \varphi, \psi\rangle_M ^{(s)} := \text{ the sum of residues
of the } 1 \text{-form } \varphi(z) \psi(z)\ dz \text{ on } \CL,
\end{equation}
where $\varphi \in D_{s-2}$ and $\psi \in \mathscr{O}(\Sigma)$.

2. For $\psi \in \mathscr{O}(\Sigma)$, let $  M  _s(\psi)$ be the
linear functional of $D_{s-2}$ defined by $$\langle \varphi,
  M  _s(\psi)\rangle = \langle \varphi, \psi\rangle _M ^{(s)}, \hskip
10pt \varphi \in D_{s-2}.$$ $  M  _s: \mathscr{O}(\Sigma)\ra
(D_{s-2})^*$ is called \textbf{Morita's duality operator}.
\end{defn}

By some explicit computations of $\langle\ , \ \rangle_M ^{(s)}$
(ibid.), we obtain the following proposition.

\begin{prop}[Compare ibid. Theorem 3]\label{prop:Mm}Let $s$ be an integer.
	
(1) If $s > 0$, then $  M  _s$ induces isomorphisms
of $\RG$-representations
$$ (\mathscr{O}(\Sigma), \pi_s) \xrightarrow{\cong} ((D_{s-2}/P_{s-2})^*_b,
T_{s-2}^*)$$ and $$ (\mathscr{N}_s(\Sigma), \pi_s)
\xrightarrow{\cong}
((D_{s-2}/P^\loc_{s-2})^*_b, T_{s-2}^*).$$

(2) If $s\leqslant 0$, then $  M  _s$ induces isomorphisms of
$\RG$-representations $$ (\mathscr{O}(\Sigma)/\mathscr{N}_s(\Sigma),
\pi_s) \xrightarrow{\cong} ((D_{s-2})^*_b, T_{s-2}^*).$$

We still denote the isomorphisms in (1) and (2) by $  M  _s$.
\end{prop}

\subsection{The duality operator $I_s$}

We define a continuous linear operator $I_s$ from $(B_{-s})^*_b$ to
$\mathscr{N}_s(\Sigma)$ (see \textsection \ref{sec:Isigma})
\begin{equation} \label{eq:Is}
I_s(\xi) (Z) :=  \langle \varphi^{(-s)}_{Z}, \xi \rangle.
\end{equation}

\begin{thm}\label{thm:commutative}
If $s $ is a positive integer, then we have a commutative diagram:
$$
\begin{CD}
(\mathscr{N}_s(\Sigma), \pi_s) @>{(s-1)!\   M  _s}>> ((D_{s-2}/P^\loc_{s-2})^*_b, T_{s-2}^*)\\
@AA{I_s}A @A{S_{s-2}^*}AA  \\
((B_{-s})^*_b, T_{-s}^*) @= ((D_{-s})^*_b, T_{-s}^*)
\end{CD}
$$
\end{thm}
\begin{proof}
Let $i$ be
the inclusion: $B _{-s} \hookrightarrow D_{-s}$. Then $i^*: (D_{-s})^*_b \ra  (B_{-s})^*_b$ is surjective due to the Hahn-Banach
Theorem. According to Theorem \ref{thm:Iisomorphism}, Proposition
\ref{prop:Mm} (1) and Proposition \ref{prop:Casselman}, the maps $I_s$,
$(s-1)!   M  _s$ and $S_{s-2}^*$ in the diagram are all isomorphisms
of $\RG$-representations. Therefore it suffices to prove the
commutativity of the following diagram: $$ \xymatrix@C=-1pc@R=2pc{
(\mathscr{N}_s(\Sigma), \pi_s)\ar[rr]^(.4){(s-1)!\   M  _s}& &
((D_{s-2}/P^\loc_{s-2})^*_b,
T_{s-2}^*)\ar@{->>}[ld]^(.35){ i^*\circ (S_{s-2}^{-1})^*}\\
 & ((B_{-s})^*_b,
 T_{-s}^*)\ar[ul]^(.6){I_s}}$$

We define $\xi_\infty\in (B_{-s})^*$ by $\langle \varphi^{(-s)}_Z,
\xi_\infty\rangle = \varphi^{(-s)}_Z(\infty) = 1$, then
$I_s(\xi_{\infty})(Z) = 1$ by definition (\ref{eq:Is}).
Since $\pi_s(g) 1$, for all $g\in \RG$, topologically spans
$\mathscr{N}_s$, we are reduced to proving 
\begin{equation}
\label{eq:commutativity1} (s-1)!\ (S_{s-2}^{-1})^* \circ   M  _s (1) =
\xi_{\infty}.
\end{equation}

For any $Z\in \Sigma$, we have $S_{s-2}
\left(\varphi^{(-1)}_Z\right) = (s-1)! \varphi^{(- s)}_Z,$ hence
\begin{align*}
& \langle \varphi^{(- s)}_Z, (s-1)!\ (S_{s-2}^{-1})^* \circ   M  _s
(1) \rangle \\ =\ & \langle (s-1)! S_{s-2}^{-1}(\varphi^{(- s)}_Z),
  M  _s
(1) \rangle \\
 =\ &\langle \varphi^{(-1)}_Z,  1\rangle_M ^{(s)}\\
 =\ & \Res_\infty (Z-z)^{-1} dz\\  = & 1 \\
 =\ & \langle \varphi^{(- s)}_Z,\xi_{\infty}\rangle.
\end{align*}
Since  $\varphi^{(-s)}_{Z}$, for all $Z\in \Sigma$, topologically
spans $B_{-s}$, (\ref{eq:commutativity1}) follows.
\end{proof}

\delete{
\begin{rem}
(1) By Schur's lemma, $  M  _s$ equals to $S_{s-2}^* \circ I_s^{-1}$
up to a constant; the constant is $(s-1)!$ according to Theorem
\ref{thm:commutative}.

(2) Let $z_0\in F$, we define $\xi_{z_0}\in (B_{-s})^*$ to be the
pre-image of $\psi_{z_0}$ under $I_s$. We also check the equality
\begin{equation}
\label{eq:commutativity2}\langle (Z-z)^{-1}, (s-1)!\   M  _s
(\psi_{z_0}) \rangle = \langle (Z-z)^{-1}, S_{s-2}^*
(\xi_{z_0})\rangle,
\end{equation}
for all $Z\in \Sigma$. Both sides equal to $(s-1)!\ (Z-z_0)^{-s}$.
\end{rem}
}

If $s\leqslant 0$, then $I_s: (B_{-s})^*_b \ra \mathscr{N}_s(\Sigma)$ is
an isomorphism between two $(-s+1)$-dimensional
$\RG$-representations.

\section{Concluding remarks}


Professor P. Schneider pointed out that the $p$-adic Siegel upper
half-space $\mathbf{\Sigma}$ was constructed in M. van der Put and
H. Voskuil's paper \cite{PutVoskuil} as the symmetric space
associated to the symplectic group $\RG = \Sp(2n, F)$. In fact, if
we let $\RP^-$ denote the transpose of $\RP$, and $\mathbf{G}$,
$\mathbf{U}$ and $\mathbf{P^-}$ the $F$-rigid analytifications of
$\RG$, $\RU$ and $\RP^-$, respectively, then $\mathbf{\Sigma}$ can be
realized as the complement of all the $\RG$-translations of
$(\mathbf{G} - \mathbf{U}\cdot \mathbf{P^-})/\mathbf{P^-}$ in
$\mathbf{G}/\mathbf{P^-}$. However, the construction of the affinoid
covering using the Bruhat-Tits building in \cite{PutVoskuil} is
different from ours.

We claim that this observation enables us to generalize most of the
constructions and results in this article to split reductive
groups.

\end{document}